\documentclass[a4paper,11pt,reqno]{amsart}

\usepackage[latin1]{inputenc}
\usepackage{graphicx}
\usepackage{amsmath,amsthm,latexsym,amscd,amssymb,MnSymbol}
\usepackage[all]{xy}
\usepackage{a4}
\usepackage{hyperref}
\usepackage{url}
\usepackage{fancyhdr}
\usepackage{amscd}
\usepackage[usenames,dvipsnames]{xcolor}
\numberwithin{equation}{section}

\definecolor{light}{gray}{.95}

\newtheorem{prop}{Proposition}[section]
\newtheorem{lem}[prop]{Lemma}
\newtheorem{theorem}[prop]{Theorem}

\theoremstyle{definition}
\newtheorem{dfn}[prop]{Definition}

\theoremstyle{remark}
\newtheorem{rem}[prop]{Remark}
\newtheorem{construction}[prop]{Construction}
\newtheorem{rems}[prop]{Remarks}

\newtheorem{properties}[prop]{Properties}
\newtheorem{notation}[prop]{Notation}

\def\dim{\mathop{\rm dim}}

\def\Hom{\mathop{\rm Hom}}

\def\id{\mathop{\rm id}}

\newcommand{\cK}{{\mathcal K}}

\def\phi{\varphi}

\def\R{{\mathbin{{\mathbb {R}}}}}
\def\Z{\mathbin{\mathbb {Z}}}

\def\C{\mathbb{C}}

\def\cA{\mathcal A}
\def\ccA{{\mathfrak A}}
\def\cL{{\mathcal L}}

\def\cT{{\mathcal T}}
\def\cL{{\mathcal L}}

\def\N{\mathcal{V}}

\def\KFP{{the weak (KFP) property}} 
\def\SKFP{{property  (KFP)}}

\def\ind{\mathop{\rm ind}}
\def\dim{\mathop{\rm dim}}

\def\cQ{{\mathcal {Q}}}

\newcommand\Di{D\kern-7pt/}

\definecolor{light}{gray}{.95}

\definecolor{olive}{cmyk}{0.64,0,0.95,0.4}

{\par \vspace{0cm}\ \\ \noindent 
\%\color{blue}\% }{\%\color{black}\%\vspace{0cm}}

\let\oldmarginpar\marginpar
\renewcommand\marginpar[1]{\-\oldmarginpar[\raggedleft\footnotesize #1]%
{\raggedright\footnotesize #1}}

\newcommand{\morph}[2]{\delta^{#1}_{#2}}

\begin{document}

\title[Bivariant $K$-theory with $\R/\Z$-coefficients and rho classes]{Bivariant $K$-theory with $\R/\Z$-coefficients and rho classes of unitary representations}

\author{Paolo Antonini, Sara Azzali and Georges Skandalis}
\thanks{Paolo Antonini has received funding from
the \emph{European Research Council} (E.R.C.) under European Union's Seventh Framework Program
(FP7/2007-2013), ERC grant agreement No. 291060; Georges Skandalis is funded by ANR-14-CE25-0012-01. We thank these institutions for their support.}
\maketitle
\begin{abstract}
We construct equivariant $KK$-theory with coefficients in $\R$ and $\R/\Z$ as suitable inductive limits over ${\rm II}_1$-factors. We show that the Kasparov product, together with its usual functorial properties, extends to $KK$-theory with real coefficients.

Let $\Gamma$ be a group. We define a $\Gamma$-algebra $A$  to be \emph{$K$-theoretically free and proper} (KFP) if  the group trace ${\bf tr}$ of $\Gamma$ acts as the unit element in $KK^\Gamma_\R(A,A)$. We show that free and proper $\Gamma$-algebras (in the sense of Kasparov) have the (KFP) property. Moreover, if $\Gamma$ is torsion free and satisfies the $KK^\Gamma$-form of the Baum-Connes conjecture, then every $\Gamma$-algebra satisfies (KFP). 

If $\alpha:\Gamma\to U_n$ is a unitary representation and $A$ satisfies property (KFP), we construct in a canonical way a rho class $\rho_\alpha^A\in KK_{\R/\Z}^{1,\Gamma}(A,A)$. This construction generalizes the Atiyah-Patodi-Singer $K$-theory class with $\R/\Z$ coefficients associated to $\alpha$.  

\end{abstract}

\setcounter{tocdepth}{1}

\section*{Introduction}

Let $V$ be a closed manifold with fundamental group $\Gamma$. The symmetric index of an elliptic pseudodifferential operator $D$ on $V$ is an element of the $K$-theory of the group $C^*$-algebra $\ind_\Gamma(D)\in K_0(C^*\Gamma)$.

Atiyah's $L^2$-index theorem for covering spaces (\cite{Ati}) expresses the triviality of the group trace on this element: it states that the image of the index class $\ind_\Gamma(D)$ by the trivial representation of $\Gamma$, which is the ordinary index of $D$, coincides with the image of $\ind_\Gamma(D)$ by the group trace of $\Gamma$, 
\emph{i.e.} the von Neumann index of the associated $\Gamma$-invariant operator acting on the universal cover $\widetilde V$.

This property of the group trace in Atiyah's theorem is equivalent  to the fact that the Mishchenko bundle with fibre a ${\rm II}_1$-factor can be trivialized. This plays a crucial role in the construction of secondary invariants.

Let  $\alpha: \Gamma\to U_n$ be a unitary representation. Atiyah, Patodi and Singer constructed a class $[\alpha]_{{\rm APS}}$ in the $K$-theory of $V$ with $\R/\Z$-coefficients in a way that the pairing of $[\alpha]_{{\rm APS}}$ with the $K$-homology class $[D]$ is equal in $\R/\Z$ to the reduced rho invariant \cite{APS2, APS3}. 

We showed in \cite{AAS}, to which we refer the reader for a more comprehensive literature, that the class $[\alpha]_{{\rm APS}}$ can be given by a purely $K$-theoretical construction, using von Neumann algebras.  In this description, $[\alpha]_{{\rm APS}}$ is a relative class, represented by two bundles which become isomorphic when twisted by a bundle with fiber a ${\rm II}_1$-factor. In a sense we showed that   $[\alpha]_{{\rm APS}}$ can be constructed as a secondary invariant built as a consequence of the triviality of the trace in Atiyah's $L^2$-index theorem.

\medskip
In the present paper, we examine general cases for which the same triviality of the trace occurs and associate to them secondary invariants in an appropriate $KK$-theory with $\R/\Z$ coefficients.

To do so, we first have to give a general definition of $KK$-theory with $\R$ and $\R/\Z$ coefficients.

It is quite clear that the $KK$-theory with real coefficients should naturally be defined using ${\rm II}_1$-factors; in other words, one is tempted to put $KK_\R(A,B)=KK(A,B\otimes M)$ where $M$ is a ${\rm II}_1$-factor. We immediately fall into the question: which ${\rm II}_1$-factor should be chosen? If $A$ and $B$ are in the bootstrap category, the group $KK(A,B\otimes M)$ is equal to $\operatorname{Hom}(K_*(A)\to K_*(B)\otimes \R)$ whence does not depend on the ${\rm II}_1$-factor $M$. But in general, this group depends (a priori) on $M$. There is of course a ``minimal choice'' for $M$, namely the hyperfinite ${\rm II}_1$-factor $R$, but we need here factors big enough to contain the group von Neumann algebra of our $\Gamma$: then it is not reasonable in this discussion to assume $\Gamma$ to be amenable, since the typical $\Gamma$ is a dense subgroup of $U_n$. On the opposite, Ozawa (\cite{Oz})  proved that there is no ``maximal choice'' for $M$.

Our solution is to define $KK_\R(A,B)$ as the \emph{inductive limit of}  $KK(A,B\otimes M)$  \emph{over all  $\,{\rm II}_1$-factors $M$}. It is not too difficult to give a sense to this inductive limit. One extends also very easily the Kasparov product with all its functorial properties and associativity to the $KK_\R$-theory (using von Neumann tensor products).

We then define $KK_{\R/\Z}$ as the inductive limit of $KK(A,B\otimes  {\rm Cone}(\C\to M))$. The mapping cone exact sequence yields the Bockstein change of coefficients sequence.

In the same way, one defines all the equivariant $KK$-theories with $\R$ and $\R/\Z$ coefficients, letting the group (or more generally a group-like object) act trivially on the ${\rm II}_1$-factors.

\medskip A tracial state $\tau$ on a $C^*$-algebra $D$ may be though of as a generalized morphism to a ${\rm II}_1$-factor: in fact, there is a ${\rm II}_1$-factor $M$ and a tracial morphism $D\to M$, so that we obtain a class $[\tau]$ in $KK_{\R}(D,\C)$. Using Kasparov product, we then obtain morphisms $\tau_*:KK(A,B\otimes D)\rightarrow KK_{\R}(A,B)$ for every $C^*$-algebras $A$ and $B$.

Let now $\Gamma$ be a group. A tracial state on $C^*\Gamma$ defines an element in $KK_{\R}(C^*\Gamma,\C)$. This group bares a ring structure using the coproduct of $\Gamma$, and is actually equal to the equivariant group $KK_{\R}^{\Gamma}(\C,\C)$.

\medskip 
 Denote by ${\bf tr}$ the group trace of $\Gamma$ and $[{\bf tr}]^\Gamma$ its class, which is an idempotent of the ring $KK_{\R}^{\Gamma}(\C,\C)$. As all the Kasparov groups $KK_\R^\Gamma(A,B)$ are modules over this ring, a natural question is: what is the image of this idempotent on these $KK_{\R}^{\Gamma}(\C,\C)$ modules?

We will be interested on those $C^*$-algebras $A$ endowed with an action of $\Gamma$ (in short, such an $A$ is called a $\Gamma$-algebra), for which the element $[{\bf tr}]^\Gamma$ acts as a unit element in $KK^\Gamma_\R(A,A)$. This means that $1_A^\Gamma\otimes_\C[{\bf tr}]^\Gamma=1_{A,\R}^\Gamma\in KK^\Gamma_\R(A,A)$. When this holds, we say that $A$ satisfies \SKFP.

This definition is inspired by the commutative case where the deck group $\Gamma$ of a Galois covering $\widetilde V$ of a closed manifold $V$ acts on $C_0(\widetilde V)$. Atiyah's theorem in \cite{Ati} can be interpreted as stating that the $\Gamma $-algebra $C_0(\widetilde V)$ satisfies  \SKFP.

In this commutative case,  \SKFP\ comes from the fact that the action of $\Gamma$ on $\widetilde V$ is free and proper. This is why algebras satisfying this condition are thought of as being in a sense $K$-theoretically free and proper.

\medskip
Let $[\alpha]^\Gamma\in KK^\Gamma(\C,\C)$ be the class of a unitary representation $\alpha:\Gamma \to U_n(\C)$ in the Kasparov representation ring. Since the regular representation of $\Gamma$ absorbs, we have $[\alpha]^\Gamma\otimes [{\bf tr}]^\Gamma=n.[{\bf tr}]^\Gamma$.  If $A$ is a $\Gamma$-algebra with  \SKFP, then the difference $1^\Gamma_A\otimes [\alpha]^\Gamma-n 1^\Gamma_A\in KK^\Gamma(A,A)$ vanishes in $KK_\R^\Gamma(A,A)$, and thus has a lift under the boundary map 
$$
\partial:KK_{\R/\Z}^{1,\Gamma}(A,A)\longrightarrow KK^\Gamma(A,A) 
$$
of the Bockstein sequence.
Our main construction, Theorem \ref{Strongrhoconstruction}, is in fact a transgression of $[\alpha]^\Gamma$, {\it i.e.} a canonical lift $\rho_\alpha^A\in KK_{\R/\Z}^{1, \Gamma}(A,A)$ that we call the \emph{rho class} associated to $A$ and $\alpha$. It generalizes the APS class and it is additive with respect to direct sum of representations; its behavior with respect to tensor products is also easy to describe. Furthermore $\rho_\alpha^A$ is functorial with respect to the algebra $A$ and, more generally, with respect to the action of the $KK^\Gamma$-groups by Kasparov product. 

\medskip To complete the construction, it is then natural to exhibit classes of algebras satisfying \SKFP.
\begin{itemize}
\item The first example comes from the \emph{Mishchenko bundle}, {\it i.e.} the already mentioned cocompact regular covering space $\widetilde V$: $C_0(\tilde V)$ satisfies \SKFP. In this case, we compare the rho class $\rho^{C_0(\tilde V)}_\alpha$ with the element $[\alpha]_{APS}$.
\item Based on the example of the Mishchenko bundle, we prove that free and proper $\Gamma$-algebras in the sense of Kasparov satisfy \SKFP. 
\item Moreover, \SKFP\ is obviously invariant under $KK^\Gamma$-(sub)equivalence. It follows that if $\Gamma$ is torsion free and satisfies the $KK^\Gamma$-form of the Baum-Connes conjecture, then every $\Gamma$-algebra satisfies \SKFP.
\item If $\Gamma $ is torsion free and has a $\gamma$ element in the sense of Kasparov, then we can construct in this way the $\gamma$ part of $\rho^\C_\alpha$, and therefore of $\rho^A_\alpha$ for any $\Gamma$-algebra $A$.
\item Using a result of Guentner-Higson-Weinberger (\cite{Weinberger et al}), we moreover deduce a construction of $\rho^\C_\alpha$ whenever $\alpha(\Gamma)$ has no torsion.
\end{itemize}

\medskip Finally, we state a weakening of \SKFP, called the \KFP, saying that the image of $[{\bf tr}]$ by Kasparov's descent morphism $j_\Gamma$ in $KK_\R(A\rtimes \Gamma,A\rtimes \Gamma)$ is the unit element. Under this condition, we construct a weaker rho class  $\hat \rho_\alpha ^A\in KK^1_{\R/\Z}(A\rtimes \Gamma,A\rtimes \Gamma)$. We show that  these two constructions are related:  if $A$ satisfies  \SKFP, then $\hat \rho_\alpha ^A=j_\Gamma(\rho_\alpha ^A)$.

\medskip Here is a summary of our paper:\begin{itemize}
\item In the first section, we construct the $KK$-theory with coefficients $\R$ and $\R/\Z$. We also discuss a few other $K$-theoretic constructions as the torus algebra of a pair of maps $\varphi_i:A\to B$ and the corresponding six term exact sequence in (equivariant) $KK$-theory.
\item In the second section we introduce property (KFP) and construct the rho element associated to a finite dimensional unitary representation.
\item In the third section we discuss examples of $\Gamma$-algebras with property (KFP).
\item Finally, in section 4, we introduce a weakening of property (KFP) and construct the corresponding weak rho class.
\end{itemize}
        
\section{$K$-theoretic constructions}

\subsection{Some conventions}

We will freely use Kasparov's $KK$-groups and notation from \cite{Kas1, Kas2}.

Let us fix a few conventions that will be used throughout the text:

\begin{itemize}
\item In what follows, by trace on a $C^*$-algebra we will mean a finite (positive) trace. 

\item All traces on von Neumann algebras that we consider, and all morphisms are assumed to be normal.
\item The suspension of a $C^*$-algebra $A$ is $SA=C_0((0,1),A)$.

\item If $A$ is a $C^*$-algebra, we denote by $1_A$ the unit element of the ring $KK(A,A)$. If $A$ is a $\Gamma$-algebra, the unit element of the ring $KK^\Gamma(A,A)$ will be denoted by $1_A^\Gamma$. If $u:A\to B$ is an equivariant morphism, we will denote by $[u]^\Gamma\in KK^\Gamma(A,B)$ its class: $[u]^\Gamma=u^*(1^\Gamma_B)=u_*(1^\Gamma_A)$.

\item More generally, if $E$ is an equivariant Hilbert $B$-module endowed with a left action of $A$ through an equivariant morphism $A\to \mathcal K (E)$, we will denote by $[E]^\Gamma$ the corresponding class in $KK^\Gamma(A,B)$.

\item  In what follows, $\otimes$ means minimal tensor product. Note however that we could do the construction below using maximal tensor products (and normal tensor products wherever von Neumann algebras are involved) as well.

However, if $\Gamma$ is a locally compact group, $C^*\Gamma$ will denote the \emph{full} group $C^*$-algebra of $\Gamma$.

\item If $u\in KK^\Gamma(A_1,B_1)$ and $v\in KK^\Gamma(A_2,B_2)$, we will denote by $u\otimes v\in KK^\Gamma(A_1\otimes A_2,B_1\otimes B_2)$ their external $KK$-product.

\item In particular, if $x\in KK^\Gamma(A,B)$ and $D$ is a $\Gamma$-algebra $1^\Gamma_D\otimes x\in KK^\Gamma(D\otimes A,D\otimes B)$ is the element noted $\tau_D(x)$ in \cite{Kas1}.

\end{itemize}

\begin{rem}
 In this paper, we deal with $KK$-theory in connection with  von Neumann algebras. A von Neumann algebra is (almost) never a separable $C^*$-algebra.

On the other hand, recall from \cite[Remark 3.2]{Sk2} that if $A$ and $B$ are $C^*$-algebras with $A$ separable then $KK(A,B)$ is the direct limit over the separable subalgebras of $B$. One then should define $KK(A,B)$ with $A$ not separable as the projective limit over all separable subalgebras of $A$. Kasparov's $KK$-group maps to this new $KK(A,B)$ and, with this definition, the Kasparov product $\otimes_D:KK(A,D)\times KK(D,B)\to KK(A,B)$ is defined for any triple of $C^*$-algebras with no assumptions of separability of any kind, as well as the more general $\otimes_D:KK(A_1,B_1\otimes D)\times KK(D\otimes A_2,B_2)\to KK(A_1\otimes A_2,B_1\otimes B_2)$.
\subsection{Torus algebra}\label{torus}
Let $A,B$ be $C^*$-algebras and $\varphi_0,\varphi_1:A\to B$ two $*$-homomorphisms. Define the corresponding \emph{torus algebra} to be  $$\cT(\varphi_0,\varphi_1)=\{(a,f)\in A\times B[0,1];\ f(0)=\varphi_0(a) \ \hbox{and}\ f(1)=\varphi_1(a)\}.$$
\end{rem}

We will use the following straightforward functorial properties:
\begin{enumerate}
\item if $\varphi_0,\varphi_1:A \longrightarrow B$ then $\mathcal{T}(\varphi_0,\varphi_1)\otimes D =\mathcal{T}({\varphi_0\otimes \operatorname{id}_D\,,\,\varphi_1\otimes  \operatorname{id}_D})$.
\item Every $f:D\longrightarrow A$ induces $f^*:\mathcal{T}({\varphi_0\circ f,\varphi_1 \circ f})\longrightarrow\mathcal{T}(\varphi_0,\varphi_1) $.
\item  Every $f:B\longrightarrow D$ induces $f_*:\mathcal{T}(\varphi_0,\varphi_1)\longrightarrow \mathcal{T}({f\circ \varphi_0,f \circ \varphi_1})$.
\item Let $\varphi_0,\varphi_1,\psi_0,\psi_1:A\to B$; if $\varphi_i$ is homotopic to $\psi_i$, the corresponding torus algebras $\mathcal{T}(\varphi_0,\varphi_1)$ and $\mathcal{T}(\psi_0,\psi_1)$ are homotopy equivalent. More precisely, if $\Phi_0,\Phi_1:A\to B[0,1]$ are morphisms joining $\varphi_0$ to $\psi_0$ and $\varphi_1$ to $\psi_1$ ({\it i.e} such that $(\Phi_i(a))(0)=\varphi_i(a)$ and $(\Phi_i(a))(1)=\psi_i(a)$ for all $a\in A$), we construct a homotopy equivalence $\mathcal{T}(\varphi_0,\varphi_1)\to \mathcal{T}(\psi_0,\psi_1)$ mapping $(a,f)$ to $(a,g)$, where $$g(t)=
\begin{cases}
\Phi_0(1-3t)&\hbox{if}\quad 0\le t\le 1/3\\
f(3t-1)&\hbox{if}\quad 1/3\le t\le 2/3\\
\Phi_1(3t-2)&\hbox{if}\quad 1/3\le t\le 2/3.
\end{cases}
$$
\end{enumerate}

We have an exact sequence $\xymatrix{0 \ar[r] & SB\ar[r]^{h\;\;\;\;} &\mathcal T(\varphi_0,\varphi_1)\ar[r]^{\;\;\;\;p} & A\ar[r]& 0}$ (with completely positive lifting), where $p:\mathcal T(\varphi_0,\varphi_1)\to A$ is the morphism $(a,f)\mapsto a$ and $h:SB\to \mathcal T(\varphi_0,\varphi_1)$ is the map $f\mapsto (0,f)$.

 The corresponding connecting map is $[\varphi_0]-[\varphi_1]$. We then have:

\begin{prop}\label{torexS}
Let $\Gamma$ be a discrete group, and let $A,B,C,D$ be $\Gamma$-algebras. Let $\varphi_0,\varphi_1:A\to B$ be $\Gamma$-equivariant morphisms. We have six term exact sequences $$\xymatrix{KK^\Gamma(C,D\otimes SB)\ar[r]^{h_*\;\;\;}&KK^\Gamma(C,D\otimes \mathcal T(\varphi_0,\varphi_1))\ar[r]^{\;\;\;p_*}&KK^\Gamma(C,D\otimes A)\ar[d]^{(\varphi_0)_*-(\varphi_1)_*}\\
KK^\Gamma(C,D\otimes SA)\ar[u]^{(\varphi_0)_*-(\varphi_1)_*}&KK^{1,\Gamma}(C,D\otimes \mathcal T(\varphi_0,\varphi_1))\ar[l]_{p_*\;\;}&KK^\Gamma(C,D\otimes B).\ar[l]_{\;\;\;\;\; \;\;h_*}}\ $$

There is also an analogous exact sequence with reversed arrows for $A,B$ and $\cT(\varphi_0,\varphi_1)$ on the left hand side.
\begin{proof}
Although this follows from \cite{Kas1, Sk} (at least in the non equivariant setting), we wish to outline here that this is a Puppe type exact sequence and holds automatically for every homotopy functor (see \cite{CuSk}). In particular, it holds for the equivariant $KK^\Gamma$-theory as we will use it.

\begin{description}
\item [Exactness at $A$] This is in a sense just tautological; for $x\in  KK^\Gamma(C,D\otimes A)$, we have $(\varphi_0)_*(x)=(\varphi_1)_*(x)$ if and only if these elements are homotopic \emph{i.e.} if $x$ is in $p^*(KK^\Gamma(C,D\otimes \mathcal T(\varphi_0,\varphi_1)))$.

\item [Exactness at $\cT$] The mapping cone of $p$ is $\cT(\varphi_0\circ e_0,\varphi_1\circ e_0)$ where $e_0:C_0([0,1),A)\rightarrow A$ is the evaluation at $0$. As $e_0$ is homotopic to the evaluation at $1$, which is the $0$ map, this torus algebra is homotopy equivalent to $\cT(\varphi_0\circ e_1,\varphi_1\circ e_1)=SB\oplus C_0([0,1),A)$ whence to $SB$. More precisely, the inclusion of the kernel $SB$ of $p$ in the mapping cone of $p$ is a homotopy equivalence. Use then the Puppe sequence which holds in full generality.

\item [Exactness at $B$] The mapping cone of $p$ is $\cT(\widetilde \varphi_0,\widetilde \varphi_1)$ where $\widetilde \varphi_i:SA\to C_0([0,1),B)$ are $f\mapsto \varphi_i\circ f$. These maps are again homotopic to $0$, whence $\cT(\widetilde \varphi_0,\widetilde \varphi_1)$ is homotopy equivalent to $SA$. \qedhere
\end{description}
\end{proof}
\end{prop}

\begin{rems}\label{commuT}
\begin{enumerate}
\item \label{MoritaT} Let $\varphi_0,\varphi_1:A\to B$ be morphisms and $u_0,u_1\in B$ unitaries. Then we have a natural Morita equivalence between the torus algebras $\mathcal T(\varphi_0,\varphi_1)$ and $\mathcal T({\rm Ad}_{u_0}\circ \varphi_0,{\rm Ad}_{u_1}\circ \varphi_1)$. Indeed the space $E=\{(a,f)\in A\times B[0,1];\ f(0)=u_0\varphi_0(a);\ f(1)=u_1\varphi_1(a)\}$ is the desired Morita equivalence bimodule. 

If $[E]$ denotes the associated element in $KK(\mathcal T({\rm Ad}_{u_0}\circ \varphi_0,{\rm Ad}_{u_1}\circ \varphi_1),\mathcal T(\varphi_0,\varphi_1))$, we have 
$$p_*[E]=[p'_*]\ ,\;\; \ \ (h')^*[E]=[h]$$ where $p:\mathcal T(\varphi_0,\varphi_1)\to A$ and $p':\mathcal T({\rm Ad}_{u_0}\circ \varphi_0,{\rm Ad}_{u_1}\circ \varphi_1)\to A$ denote the evaluation morphisms $(a,f)\mapsto a$, and $h:SB\to \mathcal T(\varphi_0,\varphi_1)$ and $h':SB\to \mathcal T({\rm Ad}_{u_0}\circ \varphi_0,{\rm Ad}_{u_1}\circ \varphi_1)\to A$ denote the morphisms $f\mapsto (0,f)$.

\item \label{T-to-Z}
Let $A_0$ and $A_1$ be $C^*$-algebras with $\varphi_0:A_0\to B$ and $\varphi_1:A_1\to B$ be two morphisms. Recall that the associated double cylinder is defined by $Z(\varphi_0,\varphi_1)=\{(a_0,a_1,f)\in A_0\times A_1\times B[0,1];\ f(0)=\varphi_0(a_0) \ \hbox{and}\ f(1)=\varphi_1(a_1)\}$.  Given unitaries $u_0,u_1\in B$, there is a canonical Morita equivalence $Z(\varphi_0,\varphi_1)$ and $Z({\rm Ad}_{u_0}\circ \varphi_0,{\rm Ad}_{u_1}\circ \varphi_1)$. In fact, we may put $A=A_0\times A_1$ and $\psi:A\to B$, $\psi_i(a_0,a_1)=\varphi_i(a_i), i=0,1\}$; then   $Z(\varphi_0,\varphi_1)$ identifies with $\mathcal T(\psi_0,\psi_1)$.

\end{enumerate}
\end{rems}

\subsection{$KK$-theory with real coefficients}

It is natural to try to define the $KK$-theory with real coefficients as $KK(A,B\otimes M)$, where $M$ is a $\rm{II}_1$-factor. On the other hand, it is not clear which $M$ we should choose. 

We give here a natural construction of $KK_{\R}$ taking into account \emph{all possible} such $M$.

\bigskip
Take a separable Hilbert space $H$ and consider the set $\mathcal{F}_{\rm{II}_1}(H)$  of all $\rm{II}_1$-factors acting on $H$. Every factor $M\in \mathcal{F}_{\rm{II}_1}(H)$ is endowed with its unique tracial state. Recall the following properties:
\begin{enumerate}
\item Every finite von Neumann algebra with trace can be embedded in a trace preserving way into a $\rm{II}_1$ factor - just take for instance a free product with a $\rm{II}_1$-factor  (\cite{Vo}, \cite[Theorem 4.4]{Po}, \cite{Dy}). 

\item For the same reason, a finite number of $\rm{II}_1$ factors can be simultaneously embedded in a trace preserving way into a $\rm{II}_1$ factor - just take their (tracial) free product.

\item Given two unital morphisms  $\varphi,\psi \colon M_1\to M_2$, there exists a $\rm{II}_1$-factor $M_3$, namely (a $\rm{II}_1$-factor containing)  the corresponding von Neumann HNN extension (see \cite[\S 3]{FV}) and a unital embedding $\chi\colon M_2\to M_3$ such that $\chi \circ \varphi$ and $\chi\circ \psi$ differ by an inner automorphism and thus define the same element of $KK(M_1,M_3)$. 
\end{enumerate}

\begin{dfn}
For $A$ and $B$ separable $C^*$-algebras 
the diagram $\mathcal{F}_{\rm{II}_1}(H)\ni M \longmapsto KK(A,B\otimes M)$ with values groups has a well defined limit
that we take as a definition of:
 \begin{equation}\label{realkk}
 KK_{\mathbb{R}}(A,B):=\lim_{M\in \mathcal{F}_{\rm{II}_1}(H)}KK(A,B\otimes M).\end{equation}
 \end{dfn}

\begin{construction}\label{constrkkr}
To define this limit it is convenient to build a direct system:
 \begin{itemize}
 \item for every $M$ from $\mathcal{F}_{\rm{II}_1}(H)$ define the group $\underline{KK}(A,B)[M]$ to be the quotient of $KK(A,B\otimes M)$ with respect to the 
 subgroup of all elements $z\in KK(A,B\otimes M)$ which become zero under some 
 unital embedding $\varphi:M\rightarrow N$. It is a well defined equivalence relation by the property $(3)$ above.
 
 \item Every unital embedding $\varphi \colon  M_1\hookrightarrow M_2$ of ${\rm II}_1$-factors induces a group homomorphism 
 $\varphi_* \colon \underline{KK}(A,B)[ M_1]\longrightarrow \underline{KK}(A,B)[M_2]$ 
 sending the equivalence class of an element $[x]\in KK(A,B\otimes M_1)$ to the equivalence class of $\varphi_*[x]\in KK(A,B\otimes M_2)$. This map is well defined by  property $(2)$ and does not depend on $\varphi$ by property $(3).$ 
\end{itemize}
We define a relation of partial pre-order in $\mathcal{F}_{\rm{II}_1}(H)$ saying that $M_1\prec M_2$ if there is a tracial embedding $M_1 \hookrightarrow M_2$. The partially pre-ordered set $(\mathcal{F}_{\rm{II}_1}(H),\prec)$ is directed - by property (1) and the map $M\mapsto \underline{KK
}(A,B)[M]$ is a direct system. 
It follows immediately that the unique limit $L$ of $M\mapsto \underline{KK}(A,B)[M]$ is a limit for the diagram $M\mapsto KK(A,B\otimes M)$.
\end{construction}

\begin{rems}\label{rems1.1}
\noindent \begin{enumerate}

\item When $A$ and $B$ are in the bootstrap class, then $KK(A,B\otimes M)$ does not depend on the $\rm{II}_1$-factor $M$ and is isomorphic to $\Hom(K(A),K(B)\otimes \R)$.
In other words,  for any $\rm{II}_1$-factor $M$ the map $KK(A,B\otimes M)\longrightarrow KK_{\mathbb{R}}(A,B)$ is an isomorphism. 

\item If $A$ and $B$ are not in the bootstrap category, then there is no best choice for $M$. Indeed, the hyperfinite $\rm{II}_1$-factor is of course the smallest element of $(\mathcal{F}_{\rm{II}_1}(H),\prec)$. By \cite{Oz}, there is no biggest element in $\mathcal{F}_{\rm{II}_1}(H)$, and therefore there is no natural way to choose $M$ and define $KK_{\R}(A,B)=KK(A,B\otimes M)$. 

 Note also that if $M$ is a \emph{big} ${\rm II}_1$-factor acting on a non separable Hilbert space, then for every separable $C^*$-algebra $A$, and $x\in KK(A,B\otimes M)$, there exists (by \cite[Remark 3.2]{Sk2}) a separable subalgebra $N_1\subset B\otimes M$ and $x_1\in KK(A,N_1)$ whose image is $x$. We may then construct separable subalgebras $B_1\subset B$ and $D_1\subset M$ and hence a (weakly) separable ${\rm II}_1$-subfactor $M_1\subset M$ such that $N_1\subset B_1\otimes D_1\subset B_1\otimes M_1\subset B\otimes M$. Let then $\tilde x_1$ denote the image of $x_1\in KK_{\R}(A,B)$. If $N_2,x_2$ is another such choice, we may find a separable $N_3\subset B\otimes M$ where $x_1$ and $x_2$ coincide and therefore a separable subfactor $M_3\subset M$ such that $N_3\subset B\otimes M_3$. It follows that $\tilde x_1=\tilde x_2$. We thus obtain a natural morphism $KK(A,B\otimes M)\to KK_{\R}(A,B)$.
 
 \item If we wish to avoid the use of non separable algebras and the method of \cite{Sk2} to treat the corresponding $KK$-products, we may also replace this category of ${\rm II}_1$-factors by the category of separable $C^*$-algebras with tracial states, the morphisms in this category being the trace preserving $*$-homomorphisms. Define then $KK_\R(A,B)$ as the limit of $KK(A,B\otimes D)$ when $D$ runs over all tracial separable $C^*$-algebras (acting on a given Hilbert space). We may form an inductive system $\underline {KK}(A,B)[D,\tau]$ as the quotient of $KK(A,B\otimes D)$ by the subgroup of $x$ such that there exists a $C^*$-algebra $D_1$ with a tracial state $\tau_1$ a trace preserving morphism $j:D\to D_1$ such that $j_*(x)=0$.
 
It follows from \cite{Sk2} that the two definitions coincide.

\item The method used to show the existence of the limit in construction \ref{constrkkr} can be summarized as follows: assume $\mathcal{C}$ is a (small) category and $\mathcal{F}:\mathcal{C}\longrightarrow \operatorname{Grp}$ is a functor such that the following properties are satisfied:
\begin{itemize}
\item
for every $A,B \in \operatorname{Ob}_{\mathcal{C}}$ there exists some $C\in \operatorname{Ob}_{\mathcal{C}}$ and arrows $A\longrightarrow C$ and $B\longrightarrow C.$
\item
For every couple of arrows $f,g:A \longrightarrow B$ there exists some $C\in \operatorname{Ob}_{\mathcal{C}}$ and an arrow $h:B\longrightarrow C$ 
with $\mathcal{F}(h\circ f)=\mathcal{F}(h\circ g).$
\end{itemize}
Then we have shown that there is a natural transformation 
$\mu:\mathcal{F}\longrightarrow \underline{\mathcal{F}}$ 
such that $\underline{\mathcal{F}}$ is a directed system and $\mathcal{F}$ has a unique limit which is the unique limit of $\underline{\mathcal{F}}$.
\end{enumerate}
\end{rems}

\subsection{Traces and $KK$-theory with real coefficients}\label{traces}

A trace on a $C^*$-algebra $D$ can be thought of as a generalized morphism to an abstract $\rm{II}_1$-factor. It therefore gives rise to a natural element of $KK_{\R}(D,\C)$.

\begin{dfn}
\label{tau*}If $D$ is a separable unital $C^*$-algebra and $\tau$ is a trace on $D$, then there is a $\rm{II}_1$-factor $M$ and a finitely generated projective module $E$ on $M$ with a trace preserving morphism $f_\tau:D\to \cL(E)$ (with von Neumann dimension $\dim_M(E)=\tau(1)$); in this way, we define an element $[\tau]\in KK_{\R}(D,\C)$; using the morphism $f_\tau:D\to \cL(E)$ we also obtain for every pair $(A,B)$ of $C^*$-algebras a morphism $$\tau_*:KK(A,B\otimes D)\longrightarrow KK_{\mathbb{R}}(A,B).$$ In particular, we have  $[\tau]=\tau_*(1_D)\in KK_{\R}(D,\C)$.
\end{dfn}

\subsection{The Kasparov product}\label{KKprod}

Let $A,B,D$ be $C^*$-algebras.

 If $x\in KK_{\mathbb{R}}(A,D)$ and $y\in KK_{\mathbb{R}}(D,B)$, there exist $\rm{II}_1$-factors $M$ and $N$ such that $x$ is  the image of $x_0\in KK(A,D\otimes M)$ and $y$ is  the image of $y_0\in KK(D,B\otimes N)$. We then may form the $KK$-product and obtain an element $x_0\otimes_D y_0\in KK(A,B\otimes M\otimes N)$. We may then map the minimal tensor product $M\otimes N$ into the corresponding von Neumann tensor product and pass to the limit $KK_{\R}(A,B)$. We obtain an element $x\otimes_D y\in KK_{\mathbb{R}}(A,B)$ which only depends on $x$ and $y$ and not on $(M,x_0)$ and $(N,y_0)$.

This Kasparov product immediately extends to a product $$\otimes_D:KK_{\mathbb{R}}(A_1,B_1\otimes D)\times KK_{\mathbb{R}}(D\otimes A_2,B_2)\to KK_{\mathbb{R}}(A_1\otimes A_2,B_1\otimes B_2)$$
which has all the usual properties of the Kasparov product (bilinearity, functoriality, associativity...).

\begin{rem}
Let  $\tau$ be a trace on $D$; by naturality of the Kasparov product, the map $\tau_*:KK(A,B\otimes D)\longrightarrow KK_{\mathbb{R}}(A,B)$ is given by the product with $[\tau]\in KK_{\R}(D,\C)$.
\end{rem}

\begin{rem}
For every $\rm{II}_1$-factor $M$, the group $KK(\C,M)$ is canonically isomorphic with $\R$. It follows that the ring $KK_{\R}(\C,\C)$ is naturally isomorphic with $\R$. We deduce that $KK_{\R}(A,B)$ is naturally a real vector space - and that the $KK_{\R}$-product is $\R$-bilinear. If $\tau $ is a trace on a $C^*$-algebra $D$ and $ s\in \R_+$, then $[s\tau]=s[\tau]\in KK_{\R}(D,\C)$.
\end{rem}

\subsection{$KK$-theory with $\R/\Z$ coefficients.} For every $\rm{II}_1$-factor $M$ denote  by   $i_M:\C\to M$ the unital inclusion. For the $\R/\mathbb{Z}$-coefficients we can similarly define:

\begin{dfn} Let $A,B$ be $C^*$-algebras we put:
\begin{equation}\label{kcone}
KK_{\bullet,\mathbb{\R}/\mathbb{Z}}(A,B):=\lim_{M\in \mathcal{F}_{\rm{II}_1}(H)} KK_{\bullet+1}(A,B\otimes C_{i_M}).
\end{equation}
\end{dfn}

As above, this is a limit along the partially ordered directed set $(\mathcal{F}_{\rm{II}_1}(H),\prec)$.

The mapping cone exact sequence gives rise to a six term Bockstein change of coefficients exact sequence $$...\to KK(A,B)\to KK_{\mathbb{R}}(A,B)\to KK_{\mathbb{R}/\mathbb{Z}}(A,B)\to...$$

\begin{rem}\label{vartheta}
Let $A,B$ be $C^*$-algebras and $M$ be a ${\rm II}_1$-factor. Denote by $i_M^B:B\to B\otimes M$ the map $x\mapsto x\otimes 1_M$. We have of course $C_{i_M^B}=B\otimes C_{i_M}$ and $Z(i_M^B,i_M^B)=B\otimes Z(i_M,i_M)$. Using the split exact sequence $$0\to C_{i_M^B}\to Z(i_M^B,i_M^B) \to B\to 0,$$ we find $KK(A,B\otimes Z(i_M,i_M))=KK(A,B)\oplus KK(A,B\otimes C_{i_M})$. We thus get a morphism $\vartheta :KK(A,B\otimes Z(i_M,i_M))\to KK^1_{\R/\Z}(A,B)$. Composing with the connecting map $\partial$ of the Bockstein change of coefficients exact sequence we obtain a map $\partial \circ \vartheta =(p_0)_*-(p_1)_*:KK(A,B\otimes Z(i_M,i_M))\to KK(A,B)$, where $p_i:B\otimes Z(i_M,i_M)=Z(i_M^B,i_M^B)\to  B$ are the maps $(b_0,b_1,f)\mapsto b_i$.
\end{rem}

\subsection{$KK^{\Gamma}$ with coefficients}
The above realizations of $KK_{\R}$ and $KK_{\R/\Z}$ can be extended to the equivariant setting.

\begin{dfn}
Let $\Gamma $ be a discrete group. We define $KK_{\R}^\Gamma(A,B)$ and $KK_{\R/\Z}^\Gamma(A,B)$  to be respectively the limit of $KK^\Gamma(A,B\otimes M)$ and $KK^\Gamma(A,B\otimes C_{i_M})$ where $M$ runs over the set $\mathcal{F}_{\rm{II}_1}(H)$ of ${\rm II}_1$-factors on the separable Hilbert space $H$ - with trivial $\Gamma $ action, $i_M:\C\to M$ denotes the unital inclusion and $C_{i_M}$ its mapping cone.
\end{dfn}

This is again a direct limit along the partially preorderd directed set $(\mathcal{F}_{\rm{II}_1}(H),\prec)$.

\medskip Recall from \cite{Kas2} that Kasparov descent morphism is a natural morphism $j_\Gamma:KK^\Gamma(A,B)\to KK(A\rtimes\Gamma,B\rtimes \Gamma)$. If $A,B$ are $\Gamma$-algebras and $M$ a ${\rm II}_1$-factor (with trivial action of $\Gamma$), composing the descent morphism together with the natural morphism $(B\otimes M)\rtimes \Gamma\to (B\rtimes \Gamma)\otimes M$ we obtain a morphism $$KK^\Gamma(A,B\otimes M)\to KK(A\rtimes\Gamma,(B\otimes M)\rtimes \Gamma)\to KK(A\rtimes\Gamma,(B\rtimes \Gamma)\otimes M).$$
This map passes to the inductive limit and defines a descent morphism with real coefficients $j_{\R}^\Gamma:KK_{\R}^\Gamma(A,B)\to KK_{\R}(A\rtimes\Gamma,B\rtimes \Gamma)$.

In the same way, taking coefficients in mapping cone algebras of inclusions $i_M:\C\to M$ and then inductive limit over $M\in \mathcal{F}_{\rm{II}_1}(H)$, we obtain a descent morphism with $\R/\Z$ coefficients $j_{\R/\Z}^\Gamma:KK_{\R/\Z}^\Gamma(A,B)\to KK_{\R/\Z}(A\rtimes\Gamma,B\rtimes \Gamma)$.

The Kasparov product $$\otimes_D:KK_{\mathbb{R}}^\Gamma(A_1,B_1\otimes D)\times KK_{\mathbb{R}}^\Gamma(D\otimes A_2,B_2)\to KK_{\mathbb{R}}^\Gamma(A_1\otimes A_2,B_1\otimes B_2)$$ is defined exactly as in Section \ref{KKprod} with all the functorial properties with respect to algebras. 

Furthermore a morphism $\Gamma_1\rightarrow \Gamma_2$ induces an obvious morphism $KK_{\R}^{\Gamma_2}(A,B) \rightarrow KK_{\R}^{\Gamma_1}(A,B).$ 

Also, the mapping cone exact sequence in equivariant $KK$-theory gives rise to a six term Bockstein change of coefficients exact sequence $$...\to KK^\Gamma(A,B)\to KK^\Gamma_{\mathbb{R}}(A,B)\to KK^\Gamma_{\mathbb{R}/\mathbb{Z}}(A,B)\to...$$

\begin{rem}\label{remequ}
In this paper, we use the equivariant $KK_{\R}$- and $KK_{\R/\Z}$-theory with respect to a discrete group $\Gamma$. Of course, the same constructions can be performed equivariantly with respect to a locally compact group or a Hopf algebra (\textit{cf.} \cite{BaajSk}) or a groupoid (\textit{cf.} \cite{LeGall}).
\end{rem}

\subsection{Some remarks on $\Gamma$-algebras}

Let $\Gamma$ be a discrete group. We end these comments with some easy observations on $\Gamma$-algebras.

\subsubsection*{Trivial action and equivariant Kasparov groups}

\begin{lem}\label{lemma1}Let $\Gamma$ be a discrete group, $A$ a $\Gamma$-algebra and let $B$ be a $C^*$-algebra endowed with a the trivial action of $\Gamma $. Then there is a canonical isomorphism  $KK^{\Gamma}(A,B)\cong KK(A \rtimes \Gamma,B)$.
\begin{proof}
Note first that by \cite[Lemma 1.11]{Sk} one can take the action of $A$ (resp. $A\rtimes \Gamma$) to be non-degenerate in the cycles defining $KK^\Gamma(A,B)$ (resp. $KK(A\rtimes \Gamma;B)$). With this in mind the rest of the proof is straightforward because $\Gamma$ acts by unitary multipliers on $A\rtimes \Gamma$. 
 \end{proof} \end{lem}
 
\begin{rem}
Note that the identification in Lemma \ref{lemma1} is the composition $$KK^{\Gamma}(A,B)\overset{j_\Gamma}\longrightarrow KK(A \rtimes \Gamma,B\otimes C^*\Gamma)\overset{(\id\otimes \varepsilon)_*}\longrightarrow KK(A \rtimes \Gamma,B)$$ where $\varepsilon :C^*\Gamma\to \C$ is the trivial representation.
\end{rem}
 
The isomorphism $KK^{\Gamma}(A,B)\to KK(A \rtimes \Gamma,B)$ can be expressed as the composition of Kasparov's descent morphism $j_\Gamma:KK^{\Gamma}(A,B)\to KK(A \rtimes \Gamma,B\rtimes \Gamma)$ with the morphism $B\rtimes \Gamma=B\otimes_{max}C^*\Gamma\to B$ associated with the trivial representation.

\subsubsection*{Inner action and equivariant Morita equivalence}

\begin{lem}\label{lemma2}Let $\Gamma$ be a discrete group, $A$ a $C^*$-algebra and $g\mapsto u_g$ a morphism of $\Gamma $ to the group of unitary multipliers of $A$. Let $\beta$ be the corresponding inner action of $\Gamma $ on $A$ ($\beta(a)=u_gau_g^*$). The $\Gamma$-algebras $A$ endowed with the trivial action and $A$ endowed with the action $\alpha$ are Morita equivalent in an equivariant way.
\begin{proof}
We denote by $A_1$ and $A_\beta$ the algebra $A$ when endowed with the trivial action and the action $\beta$, respectively. Let $E$ be the Hilbert $A_1$-module $A$ endowed  with the action of $\Gamma$ given by $g.x=u_gx$. As $u_g(xy)=\beta _g(x)u_gy$, the action of $A$ by left multiplication on $E$ is equivariant: it follows that $E$ is a $\Gamma$-equivariant Morita equivalence between $A_\beta$ and $A_1$. \end{proof} \end{lem}

\subsection{Kasparov's descent and coproducts}
Let $A$ be a $\Gamma$-algebra. Denote by $\delta^A:A\rtimes \Gamma \longrightarrow (A \rtimes \Gamma)\otimes C^*\Gamma$ the coaction map. Recall that $\delta^A(au_g)=au_g\otimes u_g$ for $a\in A\subset A\rtimes \Gamma$ and $g\in \Gamma$ where $u_g$ denotes both the element of $C^*\Gamma$ and of (the multiplier algebra of) $A\rtimes \Gamma$ corresponding to $g$.

\begin{rem}
Let $B$ be a $C^*$-algebra endowed with a trivial $\Gamma$-action. Through the identification $KK^{\Gamma}(A,B)\cong KK(A \rtimes \Gamma,B)$, Kasparov's descent morphism is given by $x\mapsto (\delta^{A}_{max})^*(x\otimes_{max} 1_{C^*\Gamma}): KK(A \rtimes \Gamma,B)\to KK(A \rtimes \Gamma,B\rtimes \Gamma)=KK(A \rtimes \Gamma,B\otimes_{max} C^*\Gamma)$. Here, we denoted by $\delta^{A}_{max}:A\rtimes \Gamma \to (A \rtimes \Gamma)\otimes_{max} C^*\Gamma$ the morphism defined by the same formula as $\delta^A$.
\end{rem}

\section{K-theoretically ``free and proper'' algebras and $\rho$ class}
\label{strongrho}

\subsection{KK-theory elements with real coefficients associated with a trace on $C^*\Gamma$}

As explained in lemma \ref{lemma1}, for every $\Gamma$-algebra $A$ and every $C^*$-algebra $B$, we may identify $KK(A\rtimes \Gamma,B)$ with $KK^\Gamma(A,B)$ where $B$ is endowed with the trivial action of $B$. Replacing in this formula $B$ by $B\otimes M$ where $M$ is a ${\rm II}_1$-factor and taking the inductive limit, we find for a $\Gamma$-algebra $A$ and a $C^*$-algebra $B$ endowed with a trivial $\Gamma$ action an identification $KK_{\R}^\Gamma(A,B)=KK_{\R}(A\rtimes \Gamma,B)$ - and in the same way, $KK_{\R/\Z}^\Gamma(A,B)=KK_{\R/\Z}(A\rtimes \Gamma,B)$.

A trace $\tau$ on $C^*\Gamma$ gives rise to an element $[\tau]\in KK_{\R}(C^*\Gamma,\C)$ (Definition \ref{tau*}) and therefore an element  $[\tau]^\Gamma\in KK^\Gamma_{\R}(\C,\C)$.

\subsection{K-theoretically ``free and proper'' algebras}

We now consider a class of $\Gamma$-algebras on the $K$-theory of which ``the trace acts as the unit element''. Those are the \emph{$K$-theoretically ``free and proper'' algebras}:

\begin{dfn}
Let $\Gamma$ be a discrete group. Denote by ${\bf tr}$ its group tracial state.  We say that $A$ satisfies \SKFP\ if $1_A^\Gamma\otimes [{\bf tr}]^\Gamma$  is equal to the unit element $1_{A,\R}^\Gamma$ of the ring $KK^\Gamma_{\R}(A,A)$.
\end{dfn}

We will see in the following section important examples of algebras satisfying this property. In particular, we will see that free and proper algebras satisfy this condition (Theorem \ref{SKFPth}). Note however that, if $\Gamma$ has a $\gamma$ element in the sense of \cite{Kas2, Tu}, $K$-theoretically proper algebras should be those for which $\gamma=1$. It is not at all clear to us whether condition (KFP) implies $\gamma=1$.

On the other hand, it is worth noting that \SKFP\ implies a kind of freeness condition:   \SKFP\ cannot hold for the algebra $\C$ whenever the group $\Gamma$ has torsion. Indeed, let $\varepsilon :C^*\Gamma\to \C$ be the trivial representation. The algebra $\C$ has \SKFP\ if and only if the classes $[{{\bf tr}}]$ and $[{{\varepsilon}}]$ in  $KK_{\R}^\Gamma(\C,\C)=KK_{\R}(C^*\Gamma,\C)$ coincide.
Assume that $g\in \Gamma$ has finite order $m\ne 1$. Then the element $p=1/m(1+\gamma+\cdot \cdot \cdot +\gamma^{m-1})\in \C\Gamma$ is a projection with $[p]\otimes [{{\bf tr}}]={\bf tr}(p)=1/m$ and  $[p]\otimes[ {{\varepsilon}}]=\varepsilon(p)=1$.

\medskip
Assume $A$ satisfies  \SKFP. This means that there is a $\rm{II}_1$-factor $N$ with a morphism $\lambda:C^*\Gamma\longrightarrow N$ such that ${\rm tr}_N\circ \lambda={\bf tr}$ (where ${\rm tr}_N$ is the normalized trace of $N$) and such that the $(A,A\otimes N)$-bimodule  $A \otimes N$ where $A$ acts on the left by $a\mapsto a\otimes 1$ endowed with the action $g\mapsto g\otimes \lambda(g)$ and $g\mapsto g\otimes 1$  define the same element in $KK^\Gamma(A,A\otimes N)$.

\medskip Let us make a few comments on this definition:

\begin{properties}\begin{enumerate}
\item If $A$ satisfies \SKFP\ and $B$ is any $\Gamma$-algebra, then $A\otimes B$ satisfies \SKFP. \\ 
Indeed,  if $1_A^\Gamma\otimes [{\bf tr}]^\Gamma=1_{A,\R}^\Gamma$, then $1_A^\Gamma\otimes1_B^\Gamma\otimes [{\bf tr}]^\Gamma=(1_A^\Gamma\otimes [{\bf tr}]^\Gamma)\otimes1_B^\Gamma=1_{A\otimes B,\R}^\Gamma$.
\item Let $A,B$ be $\Gamma$-algebras. If $A$ is $KK^\Gamma$-subequivalent to $B$ and $B$ satisfies  \SKFP then so does $A$. \\ The assumption means that there exist $x\in KK^\Gamma(A,B)$ and $y\in KK^\Gamma(B,A)$ satisfying $x\otimes_By=1_A^\Gamma$. Note that $y\otimes _\C[{\bf tr}]^\Gamma=[{\bf tr}]^\Gamma\otimes _\C y=(1_B^\Gamma\otimes [{\bf tr}]^\Gamma)\otimes _B y$.
  If $B$ satisfies \SKFP, then $y\otimes_\C [{\bf tr}]^\Gamma$ is the image in $KK_{\R}^\Gamma(B,A)$ of $y$, whence $1_A^\Gamma\otimes [{\bf tr}]^\Gamma=x\otimes _B(y\otimes_\C [{\bf tr}]^\Gamma)=1_{A,\R}^\Gamma$.
\end{enumerate}
\end{properties}

\begin{rem} \label{strongprodtraces}
Let $\tau ,\tau'$ be tracial states on $C^*\Gamma$. Let $[\tau]^\Gamma$ and $[\tau']^\Gamma$ denote their classes in $KK_{\R}^\Gamma(\C,\C)$. Their Kasparov product $[\tau]^\Gamma\otimes [\tau']^\Gamma$ is then the class of the trace $\tau.\tau'=(\tau\otimes \tau')\circ \delta$. Note that $\tau.\tau'(g)=\tau(g)\tau'(g)$ for any group element $g\in \Gamma$, and in particular $\tau.{\bf tr}={\bf tr}$. If $A$ satisfies \SKFP, then $1_A^\Gamma\otimes [\tau]^\Gamma=1_{A,\R}^\Gamma$ for any tracial state $\tau$.
\end{rem} 

\begin{rem} \label{qGammainvertible}
If $A$ satisfies \SKFP, then the morphism $q_\Gamma:A\rtimes \Gamma\to A\rtimes_{red}\Gamma$ is $KK_{\R}$ invertible. \\ Indeed, given a von Neumann algebra $N$ and a morphism $\lambda:C^*\Gamma\longrightarrow N$ such that ${\rm tr}_N\circ \lambda={\bf tr}$ (where ${\rm tr}_N$ is the normalized trace of $N$), the morphism $\lambda:C^*\Gamma \to N$ factors through $C^*_{red}(\Gamma)$; it therefore defines a morphism  $\widehat{\morph{A}{\lambda}}:A\rtimes_{red} \Gamma\longrightarrow A\rtimes \Gamma\otimes N$. We obviously have $q_\Gamma^*[\widehat{\morph{A}{\lambda}}]=j_\Gamma (1_A^\Gamma\otimes [{\bf tr}]^\Gamma)$ and $(q_\Gamma)_*[\widehat{\morph{A}{\lambda}}]=j_{\Gamma,red} (1_A^\Gamma\otimes [{\bf tr}]^\Gamma)$.
\end{rem} 

We will use the quite obvious following lemma:
 
 \begin{lem}
\label{strongunitary}  Let $\alpha:\Gamma\to U_n$ be a unitary representation.  The group $\Gamma$ acts on $C(U_n)$ by left translation via $\alpha$. Let $u:U_n\to M_n(\C)$ be the identity of $U_n$ seen as an element of $C(U_n)\otimes M_n(\C)$. For $g\in \Gamma$, we denote by $\lambda_g\in C(U_n)\rtimes \Gamma$ the corresponding element.
For $g\in \Gamma$ we have $(\lambda_g\otimes \alpha_g)u=u(\lambda_g\otimes 1)$ in $\Big(C(U_n)\rtimes \Gamma\Big)\otimes M_n(\C)$.
\begin{proof}
Note that, for $f\in C(U_n)$ and $x\in U_n$ we have $(\lambda_g f\lambda_g^{-1})(x)=f(\alpha^{-1}_gx)$. Write $u=\sum a_j\otimes b_j$ where $a_j\in C(U_n)$ and $b_j\in M_n(\C)$. By definition of $u$, we have $\sum_ja_j(x)b_j=x$ for all $x\in U_n$.

We have $((\lambda_g\otimes 1) u(\lambda^{-1}_g\otimes 1))(x)=\sum_j a_j(\alpha_g^{-1}x)b_j=\alpha^{-1}_gx $, and the result follows.
\end{proof}
\end{lem}

\subsection{Gluing Kasparov bimodules}

To construct $\rho_\alpha^A$, we glue two Kasparov bimodules on a mapping cylinder in such a way that they form a Kasparov bimodule on a double cylinder. We discuss here this general construction.

Let $A,B_0,B$ be $\Gamma$-algebras and $j:B_0\to B$ an equivariant morphism. Denote by $Z_j=\{(b,f)\in B_0\times B[0,1];\ f(0)=j(b)\}$ the mapping cylinder of $j$ and $Z_{j,j}=\{(b_0,b_1,f)\in B_0\times B_0\times B[0,1];\ f(0)=j(b_0),\ f(1)=j(b_1)\}$ its mapping double cylinder. Let $\overline{\rm ev}_0:Z_j\to B_0$ and ${\rm ev}_1:Z_j\to B$ be the maps $\overline{\rm ev}_0:(b,f)\mapsto b$ and ${\rm ev}_1:(b,f)\mapsto f(1)$.

We have a split exact sequence $0\to C_j\longrightarrow  Z_{j,j}\overset{\overline{\rm ev}_1}\longrightarrow B_0\to 0$ where $C_j=\{(b,f)\in B_0\times B[0,1);\ f(0)=j(b)\}$ is the mapping cone of $j$ and $\overline{\rm ev}_1(a,b,f)=b$ - the section $\tilde \j:B_0\to Z_{j,j}$ is given by $\tilde \j(b)=(b,b,f)$ where $f(t)=j(b)$ for all $t\in [0,1]$. We thus have a group homomorphism  $\Theta:KK^\Gamma(A,Z_{j,j})\to  KK^\Gamma(A,C_j)$ which is a left inverse of the inclusion $C_j\to Z_{j,j}$ and vanishes on the image of $\tilde \j_*$.

We will identify the double cylinder $Z_{j,j}$ with the algebra $\mathcal{Z}_j=\{(x,y)\in Z_j\times Z_j;\ {\rm ev}_1(x)={\rm ev}_1(y)\}$ through the isomorphism $\Phi:Z_{j,j}\to \mathcal {Z}_j$ given by: $\Phi(a,b,f)=((a,g),(b,h))$ with $g(t)=f(t/2),\ h(t)=f(1-t/2)$ (for $t\in [0,1]$).

Given a Hilbert $Z_j$-module $E$, the Hilbert $B_0$-module $E\otimes _{\overline{\rm ev}_0} B_0$ and the Hilbert-$B$-module $E\otimes _{{\rm ev}_1} B$ are quotients of $E$. We denote by $\overline{\rm ev}_0:E\to E\otimes _{\overline{\rm ev}_0} B_0$ and ${\rm ev}_1:E\to E\otimes _{ {\rm ev}_1} B$ the quotient maps.

\begin{prop}\label{diamond}
Let $(E,F),(E',F')$ be $\Gamma$-equivariant Kasparov $(A,Z_j)$-bimodules.  \begin{enumerate}
\item Let $w\in \cL(E'\otimes _{{\rm ev}_1}B,E\otimes _{{\rm ev}_1}B)$ be a unitary equivalence between the $\Gamma$-equivariant Kaparov $(A,B)$-bimodules $({\rm ev}_1)_*(E',F')$ and $({\rm ev}_1)_*(E,F)$.
Define 
$$E\diamond_wE'=\{(\xi,\zeta)\in E\times E';\ {\rm ev}_1(\xi)=w\,{\rm ev}_1(\zeta)\}.$$ It is naturally an equivariant Hilbert-$(A,\mathcal{Z}_j)$-bimodule. The map $(F\diamond F'):(\xi,\zeta)\mapsto (F\xi,F'\zeta)$ is an (odd) element in $\cL(E\diamond_wE')$. The pair $(E\diamond_wE',F\diamond F')$ is an equivariant Kasparov $(A,\mathcal {Z}_j)$-bimodule

\item \label{homotopic_unit_equiv} If $v,w\in \cL(E'\otimes _{{\rm ev}_1}B,E\otimes _{{\rm ev}_1}B)$ are two unitary equivalences, homotopic among unitary equivalences, then the classes of $(E\diamond_vE',F\diamond F')$ and $(E\diamond_wE',F\diamond F')$ in $KK^\Gamma(A,\mathcal {Z}_j)$ coincide.

\item \label{trivialdiamond} The class $\Theta[(E\diamond_{\id }E,F\diamond F)]$ in $ KK^\Gamma(A,C_j)$ is $0$.

\item \label{diamondadd}Let $(E'',F'')$ be another $\Gamma$-equivariant Kasparov $(A,Z_j)$-bimodule. Let $w\in \cL(E'\otimes _{{\rm ev}_1}B,E\otimes _{{\rm ev}_1}B)$ and $w'\in \cL(E''\otimes _{{\rm ev}_1}B,E'\otimes _{{\rm ev}_1}B)$ be unitary equivalences between the $\Gamma$-equivariant Kaparov $(A,B)$-bimodules $({\rm ev}_1)_*(E',F')$ and $({\rm ev}_1)_*(E,F)$ and between $({\rm ev}_1)_*(E'',F')$ and $({\rm ev}_1)_*(E',F)$ respectively. Then $$\Theta[(E\diamond_{ww'}E'',F\diamond F'')]=\Theta[(E\diamond_{w}E',F\diamond F')]+\Theta[(E'\diamond_{w'}E'',F\diamond F'')].$$
\end{enumerate}
\begin{proof}
\begin{enumerate}
\item comes from the fact that $w$ is supposed to commute with $A,\Gamma$ and $wF'=Fw$. Also, it is easily seen that $\cK(E\diamond_wE')=\{(T,T')\in \cK(E)\times \cK(E');\ w(T'\otimes_{{\rm ev}_1}1)=(T\otimes_{{\rm ev}_1}1)w\}$.
\item is immediate.
\item Let $(E_t,F_t)_{t\in [0,1]}$ be a  homotopy joining  $(\overline \j\circ \overline{\rm ev}_0)_*(E,F)$ to $(E,F)$: write $(E_t,F_t)=(h_t)_*(E,F)$ where $h_t:Z_j\to Z_j$ is the map $(b,f)\mapsto (b,f_t)$ where $f_t(s)=f(st)$. Then $(E_t\diamond_{\id }E_t,F_t\diamond F_t)$ is a homotopy joining $(\tilde\j \circ \overline{\rm ev}_0)_*(E,F)$ to $(E\diamond_{\id }E,F\diamond F)$.

\item The unitary equivalences $u=\begin{pmatrix}w&0\\0&w' \end{pmatrix},v=\begin{pmatrix}0&\id \\ww'&0 \end{pmatrix}\in \cL((E'\oplus E'')\otimes _{{\rm ev}_1}B,(E\oplus E')\otimes _{{\rm ev}_1}B)$ are homotopic among unitary equivalences. By (\ref{homotopic_unit_equiv}) the classes $$[((E\oplus E')\diamond_{u}(E'\oplus E''),(F\oplus F')\diamond (F'\oplus F''))]$$ and $$[((E\oplus E')\diamond_{v}(E'\oplus E''),(F\oplus F')\diamond (F'\oplus F''))]$$  in $KK^\Gamma(A,Z_{j,j})$ coincide. But clearly  $$[((E\oplus E')\diamond_{u}(E'\oplus E''),(F\oplus F')\diamond (F'\oplus F''))]=[(E\diamond_{w}E',F\diamond F')]+ [(E'\diamond_{w'}E'',F\diamond F'')]$$ and $$[((E\oplus E')\diamond_{v}(E'\oplus E''),(F\oplus F')\diamond (F'\oplus F''))]=[(E\diamond_{ww'}E'',F\diamond F'')]+ [(E'\diamond_{id}E',F'\diamond F')].$$
The result follows from (\ref{trivialdiamond}).
\end{enumerate}
\end{proof}
\end{prop}

\subsection{Construction of $\rho_{\alpha}^A$}\label{constrrhoalphaA}

Let $A$ be a $\Gamma$-algebra satisfying \SKFP, and $\alpha$ a finite dimensional unitary representation of $\Gamma$. Denote by $[\alpha]^\Gamma\in KK^\Gamma(\C,\C)$ its $KK^\Gamma$-class.

By Remark \ref{strongprodtraces}, $1_{A,\R}^\Gamma\otimes [{\alpha}]=n1_{A,\R}^\Gamma$ where $1_{A,\R}^\Gamma$ denotes the unit element of the ring $KK^\Gamma_{\R}(A,A)$. Using the Bockstein change of coefficients exact sequence, it follows that there exists $z\in KK^{1,\Gamma}_{{\R}/\Z}(A ,A)$ whose image in $KK^\Gamma(A,A)$ is $[ {\alpha}]-n1_{A}^\Gamma$.

We now construct a class $\rho_{\alpha}^A\in KK^{1,\Gamma}_{{\R}/\Z}(A,A)$ whose image in $KK^{1,\Gamma}_{{\R}/\Z}(A ,A)$ is $1_{A,\R}^\Gamma\otimes [\alpha]^\Gamma-n1_{A, \R}^\Gamma$.

The important fact in our construction is that the element $\rho_{\alpha}^A$ is \emph{independent of all choices.}

\bigskip 
The assumption (KFP) says that there exists a $\rm{II}_1$-factor $N$ (with trivial action of $\Gamma$), with a tracial morphism $\lambda=\lambda_N:C^*\Gamma\longrightarrow N$, \textit{i.e.} a morphism such that ${\rm tr}_N\circ \lambda={\bf tr}$ (where ${\rm tr}_N$ is the normalized trace of $N$) satisfying $$1_A^\Gamma\otimes [i]=1_A^\Gamma\otimes [P_\lambda]$$ where $i=i_N:\C\to N$ is the unital morphism and $P_\lambda$  is the $\Gamma$-equivariant Hilbert $\C,N$-bimodule $N$ with the natural Hilbert $N$-module structure,  and where $g\in \Gamma$ acts by left multiplication by $\lambda(g)$.
Under the isomorphism $KK^\Gamma(\C,N)= KK(C^*\Gamma,N)$ the class $[P_\lambda]$ corresponds to the class of the morphism $\lambda$.

There exists a $\rm{II}_1$-factor $M$ containing $N$ with a morphism $C(U_n)\rtimes \Gamma\to M$ extending the morphism $\lambda :C^*\Gamma\to N\subset M$. Indeed we can take as $M$ (a $\rm{II}_1$-factor containing) the  free product of $N$ with $L^{\infty}(U_n)\rtimes \Gamma$ with amalgamation over the group von Neumann algebra of $\Gamma$.

Up to replacing $N$ by $M$, it follows, thanks to Lemma \ref{strongunitary}, that we may assume that there exists a unitary $u\in N\otimes M_n(\C)$ satisfying $(\lambda_g\otimes \alpha_g)u=u(\lambda_g\otimes 1)\in N\otimes  M_n(\C)$  for all $g\in \Gamma$.

Denote by $V_\alpha$ the equivariant $(\C,\C)$ bimodule $\C^n$ where $\Gamma$ acts through the action $\alpha$. Denote also by $V_n$ the equivariant $(\C,\C)$ bimodule $\C^n$ where $\Gamma$ acts trivially. We have a unitary equivalence $1_A\otimes u\in \cL(A\otimes P_\lambda\otimes V_n, A\otimes P_\lambda\otimes V_\alpha)$ between the $\Gamma $-equivariant Kasparov $(A,A\otimes N)$ bimodules $(A\otimes P_\lambda\otimes V_n,0)$ and $(A\otimes P_\lambda\otimes V_\alpha,0)$.

\medskip Let $i:\C\to N$ be the unital inclusion. Denote by $\bar p_0:A\otimes Z_i\to A$ and $p_1:A\otimes Z_{i}\to A\otimes N$ the maps $\id_A\otimes \overline{\rm ev}_0$ and $\id_A\otimes {\rm ev}_1$, where $\overline{\rm ev}_0:Z_{i}\to \C$ is the map $(a,f)\mapsto a$ and ${\rm ev}_1:Z_{i}\to N$ is the map $(a,f)\mapsto f(1)$.

As $1_A^\Gamma\otimes [i]=1_A^\Gamma\otimes [P_\lambda]$,  there exists an equivariant Kasparov-bimodule $(E,F)\in {\bf E}^\Gamma(A,A\otimes Z_{i})$ joining  $(A,0)$ and $(A\otimes P_\lambda,0)$. More precisely the equivariant  Kasparov bimodule $(E,F)$ is such that:\begin{itemize}
\item the induced equivariant bimodule $E\otimes _{\bar p_0}A $ is the $(A,A)$-bimodule $A$ (endowed with the natural $\Gamma$-action - and $F_0=0$);
\item  the induced equivariant bimodule  $E\otimes _{p_1}(A\otimes N)$ is the equivariant $(A,A\otimes N)$-bimodule  $A \otimes P_\lambda$.
\end{itemize}
We will say that  \emph{a ``proof of the property (KFP)'' for $A$} consists of $N,\lambda,(E,F)$ where:
\begin{itemize}
\item $N$ is a $\rm{II}_1$-factor with normalized trace ${\rm tr}_N$;
\item $\lambda:C^*\Gamma\longrightarrow N$, is a morphism such that ${\rm tr}_N\circ \lambda={\bf tr}$;
\item $(E,F)\in {\bf E}^\Gamma(A,A\otimes Z_{i})$ is a Kasparov bimodule joining  $(A,0)$ and $(A\otimes P_\lambda,0)$.
\end{itemize}

\begin{theorem}\label{Strongrhoconstruction}
\begin{enumerate}
\item Let $(E,F)\in {\bf E}^\Gamma(A,A\otimes Z_{i})$ be a Kasparov bimodule joining  $(A,0)$ and $(A\otimes P_\lambda,0)$.
The class in $KK^\Gamma(A,A\otimes Z_{{i},{i}})$ of $$(E\otimes V_{\alpha},F\otimes 1)\diamond _{\id_A\otimes u} (E\otimes V_n,F\otimes 1)$$ does not depend  on the unitary $u\in N\otimes M_n(\C)$ satisfying $(\lambda_g\otimes \alpha_g)u=u(\lambda_g\otimes 1)$ for all $g\in \Gamma$.
\item \label{indepe}The image of this class in $KK^{1,\Gamma}_{\R/\Z}(A,A)$ through the composition $$\vartheta:KK^\Gamma(A,A\otimes Z_{{i},{i}})\overset{\Theta}\longrightarrow KK^\Gamma(A,A\otimes C_{i})\longrightarrow  KK^{1,\Gamma}_{\R/\Z}(A,A)$$ does not depend:\begin{enumerate}
\item \label{indepeEF}on $(E,F)\in {\bf E}^\Gamma(A,A\otimes Z_{i})$ such that $E\otimes _{A\otimes Z_{i}}A =A$ and  $E\otimes _{A\otimes Z_{i}}A=A \otimes P_\lambda$;
\item on the $\rm{II}_1$-factor $N$ containing a tracial morphism of $\lambda_N:C^*\Gamma\to N$ and a unitary $u\in N\otimes M_n(\C)$ satisfying\begin{enumerate}
\item $1_A^\Gamma\otimes [i]=1_A^\Gamma\otimes [\lambda_N]$
\item  $(\lambda_g\otimes \alpha_g)u=u(\lambda_g\otimes 1)\in N\otimes  M_n(\C)$  for all $g\in \Gamma$.
\end{enumerate}
\end{enumerate}

\end{enumerate}
\begin{proof}
\begin{enumerate}
\item  Choose a $u\in N\otimes M_n(\C)$ such that $(\lambda_g\otimes \alpha_g)u=u(\lambda_g\otimes 1)\in N\otimes  M_n(\C)$  for all $g\in \Gamma$. Then an element $v\in N\otimes M_n(\C)$ satisfies the same property if and only if it is of the form $uw$ where $w$ is a unitary of the commutant of $\lambda (\Gamma)\otimes 1$ in $N\otimes M_n(\C)$. But the set of unitaries in the von Neumann algebra $(N\cap \lambda(G)')\otimes M_n(\C)$ is connected.
\item\begin{enumerate}
\item Let $(E',F')\in {\bf E}^\Gamma(A,A\otimes Z_{i})$ be another Kasparov bimodule joining  $(A,0)$ and $(A\otimes P_\lambda,0)$. Choose a $u\in N\otimes M_n(\C)$ such that $(\lambda_g\otimes \alpha_g)u=u(\lambda_g\otimes 1)\in N\otimes  M_n(\C)$  for all $g\in \Gamma$ and put $\check u=\id_A\otimes u$.

By Proposition \ref{diamond}(\ref{trivialdiamond}) and \ref{diamond}(\ref{diamondadd}), the image by $\Theta$ of $(E\otimes V_{\alpha},F\otimes 1)\diamond _{\check u} (E\otimes V_n,F\otimes 1)$ is opposite to the one of $(E\otimes V_n,F\otimes 1)\diamond _{\check u^*} (E\otimes V_{\alpha},F\otimes 1)$;  moreover, by Proposition \ref{diamond}(\ref{diamondadd}), the image by $\Theta$ of $$[(E'\otimes V_{\alpha},F'\otimes 1)\diamond _{\check u} (E'\otimes V_n,F'\otimes 1)]+[(E\otimes V_n,F\otimes 1)\diamond _{\check u^*} (E\otimes V_{\alpha},F\otimes 1)]$$
coincides with that of $$[(E'\otimes V_{\alpha},F'\otimes 1)\diamond _{\id} (E\otimes V_{\alpha},F\otimes 1)]+[(E\otimes V_n,F\otimes 1)\diamond _{\id} (E'\otimes V_n,F'\otimes 1)].
$$
Using again  \ref{diamond}(\ref{trivialdiamond}) and \ref{diamond}(\ref{diamondadd}), we may replace $[(E\otimes V_n,F\otimes 1)\diamond _{\id} (E'\otimes V_n,F'\otimes 1)]$ by $-[(E'\otimes V_n,F'\otimes 1)\diamond _{\id} (E\otimes V_n,F\otimes 1)]$. We end up with $$\Theta([(E,F)\diamond _{\id} (E',F')]\otimes_{\C}([V_\alpha]-[V_n])).$$ Now, since $(\hat p_0)_*[(E,F)\diamond _{\id} (E',F')]=(\hat p_0)_*[(E,F)]=1_A=(\hat p_0)_*[(E',F')]=(\hat p_1)_*[(E,F)\diamond _{\id} (E',F')]$, it follows that $\Theta([(E,F)\diamond _{\id} (E',F')])$ is in the image of an element $x\in KK^\Gamma (A,A\otimes SN)$ via the inclusion $h:SN\to C_i$. As $([V_\alpha]-[V_n]))$ is the $0$ element in $KK_\R^\Gamma(A,A)$ it follows that the image of $x\otimes _\C([V_\alpha]-[V_n]))$ in $KK^{1,\Gamma}_\R(A,A)$ is the zero element.

\item Let $i_{M,N}:N\to M$ be a morphism of ${\rm II}_1$-factors. Denote by $i_{M,N}:Z_{i_N}\to Z_{i_M}$ and $i_{M,N}:Z_{i_N,i_N}\to Z_{i_M,i_M}$ the corresponding morphism, $v=(i_{M,N}\otimes \id)(u)$. Let $(E,F)\in {\bf E}^\Gamma(A,A\otimes Z_{i_N})$ such that $E\otimes _{A\otimes Z_{i_N}}A =A$ and  $E\otimes _{A\otimes Z_{i_N}}A=A \otimes P_{\lambda_N}$. Put $(E_M,F_M)=(E,F)\otimes _{i_{M,N}}Z_{i_M}$. One checks immediately that the image by $i_{M,N}:Z_{i_N,i_N}\to Z_{i_M,i_M}$ of $[(E\otimes V_{\alpha},F\otimes 1)\diamond _{\check u} (E\otimes V_n,F\otimes 1)]$ is $$[(E_M\otimes V_{\alpha},F_M\otimes 1)\diamond _{\check v} (E_M\otimes V_n,F_M\otimes 1)].$$
\qedhere
\end{enumerate}
\end{enumerate}
\end{proof}
\end{theorem}

\begin{dfn}\label{defstrongrho}
Let $A$ a $\Gamma$-algebra satisfying \SKFP.
We denote by $\rho_{\alpha}^A\in KK^{1,\Gamma}_{\R/\Z}(A,A)$ the  image  by the map $\vartheta:KK^\Gamma(A,A\otimes Z_{i,i})\longrightarrow  KK^{1,\Gamma}_{\R/\Z}(A,A)$ of an element $[(E\otimes V_{\alpha},F\otimes 1)\diamond _{\check u} (E\otimes V_n,F\otimes 1)]\in KK^{\Gamma}(A,A\otimes Z_{i,i})$ as in Theorem  \ref{Strongrhoconstruction}.
\end{dfn}

\begin{prop}
The image of $\rho_{\alpha}^A $ under the connecting map of the Bockstein change of coefficients exact sequence $KK^{1,\Gamma}_{\R/\Z}(A,A)\to KK^\Gamma(A, A)$ is $[\alpha]^\Gamma-n1_{A}^{\Gamma}$.
\begin{proof}
Let $N,u$ and $(E,F)$ be a ``proof of property (KFP) for $A$'', as in Theorem \ref{Strongrhoconstruction}. Then, by Remark \ref{vartheta}, the image of $\rho_\alpha^A=\vartheta ([(E\otimes V_{\alpha},F\otimes 1)\diamond _{\check u} (E\otimes V_n,F\otimes 1)])$ by this connecting map is $$(\bar p_0)_*[(E\otimes V_{\alpha},F\otimes 1)\diamond _{\check u} (E\otimes V_n,F\otimes 1)])-(\bar p_1)_*[(E\otimes V_{\alpha},F\otimes 1)\diamond _{\check u} (E\otimes V_n,F\otimes 1)]).$$ But $(\bar p_0)_*[(E\otimes V_{\alpha},F\otimes 1)\diamond _{\check u} (E\otimes V_n,F\otimes 1)])=(\bar p_0)_*[(E\otimes V_{\alpha},F\otimes 1)]=1_A^\Gamma\otimes[\alpha]^\Gamma$ and $(\bar p_1)_*[(E\otimes V_{\alpha},F\otimes 1)\diamond _{\check u} (E\otimes V_n,F\otimes 1)])=(\bar p_0)_*[(E\otimes V_n,F\otimes 1)]=n.1_A^\Gamma$.
\end{proof}\end{prop}

\subsection{Properties of the construction $\rho^A_\alpha$}

\subsubsection{Dependence on $\alpha$.}
\begin{prop}
Let $A$ be a $\Gamma$-algebra satisfying \SKFP. Let $\alpha_1, \alpha_2$ be finite dimensional unitary representations of $\Gamma$. 
We have 
\begin{enumerate}
\item $\rho_{\alpha_1\oplus \alpha_2}^A=\rho_{\alpha_1}^A+\rho_{\alpha_1}^A$
\item $\rho_{\alpha_1\otimes\alpha_2}^A=\dim \alpha_1 \cdot \rho_{\alpha_2}^A+\rho_{\alpha_1}^A\otimes _\C[\alpha_2]\;=\;\dim \alpha_2\cdot \rho_{\alpha_1}^A+[\alpha_1]\otimes _\C\rho_{\alpha_2}^A.$
\end{enumerate}

\begin{proof}
 The first statement is obvious. 
 
For the second one, put $n_i=\dim \alpha_i$. Let $(E,F)$ be as in Theorem \ref{Strongrhoconstruction}. Thanks to Proposition \ref{diamond}(\ref{diamondadd}), we may write
\begin{eqnarray*}
(E\otimes V_{\alpha_1\otimes \alpha_2},F\otimes 1)\diamond   (E\otimes V_{n_1n_2},F\otimes 1)&=&\Big((E\otimes V_{\alpha_1},F\otimes 1)\diamond   (E\otimes V_{n_1},F\otimes 1)\Big)\otimes V_{\alpha_2}\\&&+\ \Big((E\otimes V_{\alpha_2},F\otimes 1)\diamond   (E\otimes V_{n_2},F\otimes 1)\Big)\otimes V_{n_1}
\end{eqnarray*}
and the result follows.
\end{proof}
\end{prop}

\subsubsection{Change of algebra. Kasparov product.} 

\begin{prop}
Let $A,B$ be $\Gamma$-algebras satisfying \SKFP. Let $f:A\to B$ be an equivariant homomorphism. For every finite dimensional unitary representation $\alpha $ of $\Gamma$, we have $f_*(\rho_\alpha^A)=f^*(\rho_\alpha^B)\in KK^{1,\Gamma}_{\R/\Z}(A,B)$.
\end{prop}

This is a particular case of the following more general result.

\begin{prop}\label{changealg}
Let $A,B$ be $\Gamma$-algebras satisfying \SKFP\ and $x\in KK^\Gamma(A,B)$. For every finite dimensional unitary representation $\alpha $ of $\Gamma$, we have $\rho_\alpha^A\otimes _Ax=x\otimes_B\rho_\alpha^B\in KK^{1,\Gamma}_{\R/\Z}(A,B)$.
\begin{proof}
 Let $(E,F)\in {\bf E}^\Gamma(A,B)$ representing $x$. 
 
 Let $N,\lambda,(E_A,F_A)$ be a ``proof of the property (KFP) for $A$''.
As $(E_A,F_A)\in {\bf E}^\Gamma(A,A\otimes Z_i)$ is a Kasparov bimodule joining $A$ to $A\otimes P_\lambda$, the Kasparov product $(E_A,F_A)\otimes _A(E,F)$ can be represented by a Kasparov bimodule $(\widetilde E,\widetilde F)\in {\bf E}^\Gamma(A,B\otimes Z_i)$ such that $\bar p_0(\widetilde E,\widetilde F)=(E,F)$ and $p_1(\widetilde E,\widetilde F)=(E\otimes P_\lambda,F\otimes 1)$. It is then immediately checked (using for instance the connexion and positivity approach of \cite{CS}) that $$[(\widetilde E\otimes V_{\alpha},\widetilde F\otimes 1)\diamond _{1_{E}\otimes  u} (\widetilde E\otimes V_n,\widetilde F\otimes 1)]$$ is a Kasparov product  $$[(E_A\otimes V_{\alpha}, F_A\otimes 1)\diamond _{1_{E}\otimes  u} ( E\otimes V_n,F\otimes 1)]\otimes _A[(E,F)]\in {\bf E}^\Gamma(A,A\otimes Z_{i,i}).$$

Therefore $\rho_A^\alpha\otimes _Ax$ is the image of the class of $$[(\widetilde E\otimes V_{\alpha},\widetilde F\otimes 1)\diamond _{1_{E}\otimes  u} (\widetilde E\otimes V_n,\widetilde F\otimes 1)]$$ in $KK^{1,\Gamma}_{\R/\Z}(A,B)$.

In the same way, let $M,\lambda_M,(E_B,F_B)$ be a ``proof of the property (KFP) for $B$''. We may of course assume (replacing $M$ and $N$ if necessary by a factor containing the free product amalgamated over the group von Neumann algebra of $\Gamma$) that $M=N$ and $\lambda _M=\lambda$.

In the same way as above, the Kasparov product $(E,F)\otimes _B(E_A,F_A)$ can be represented by a Kasparov bimodule $(\widehat E,\widehat F)\in {\bf E}^\Gamma(A,B\otimes Z_i)$ such that $\bar p_0(\widehat E,\widehat F)=(E,F)$ and $p_1(\widehat E,\widehat F)=(E\otimes P_\lambda,F\otimes 1)$. It follows that $x\otimes _B\rho_B^\alpha$ is the image of the class
 $$[(\widehat E\otimes V_{\alpha},\widehat F\otimes 1)\diamond _{1_{E}\otimes  u}  (\widehat E\otimes V_n,\widehat F\otimes 1)]$$ in $KK^{1,\Gamma}_{\R/\Z}(A,B)$.

To conclude, note that it follows from the proof of Theorem \ref{Strongrhoconstruction} (\ref{indepeEF}) that the image of $$[(\check E\otimes V_{\alpha},\check F\otimes 1)\diamond _{1_{E}\otimes  u} (\check E\otimes V_n,\check F\otimes 1)]$$ in $KK^{1,\Gamma}_{\R/\Z}(A,B)$ does not depend on $(\check E,\check F)\in {\bf E}^\Gamma(A,B\otimes Z_i)$ joining $(E,F)$ to $(E,F)\otimes P_\lambda$.
 \end{proof}
\end{prop}

\section{Examples}

\subsection{The Mishchenko bundle}

As explained in the introduction, our starting example comes from the Mishchenko bundle associated with a Galois cover of a compact manifold.

\begin{prop}\label{Miscase}
Assume that the group $\Gamma $ acts freely, properly and cocompactly on a smooth manifold $\widetilde V$.
Then $C_0(\widetilde V)$ satisfies \SKFP.
\begin{proof}
 In \cite[Proposition 5.1]{AAS}, we noted that there is a ${\rm II}_1$-factor $N$ with a morphism $\lambda :\Gamma \to N$ such that $tr_N\circ \lambda ={\bf tr}$ together with a  continuous map $\psi$ from $\widetilde V$ to the unitary group $U(N)$ of $N$ such that $\psi (g x)=\lambda(g)\psi (x)$ for all $g\in G$ and $x\in \widetilde V$. It gives rise to a unitary $w\in \cL(C_0(\widetilde V)\otimes N)$ intertwining the $\Gamma$-actions $\beta\otimes \id_N$ and $\beta\otimes \lambda$ on $C_0(\widetilde V)\otimes N$ where $\beta$ is the action of $\Gamma$ on $C_0(\widetilde V)$ by translation. Now the Hilbert module $C_0(\widetilde V)\otimes N$ endowed with the left action of $C_0(\widetilde V)$ equal to the right one and the $\Gamma$-action $\beta\otimes \id_N$ (and  $\beta\otimes \lambda$) represent respectively the images $1_{C_0(\widetilde V)}^{\Gamma}$ and $1_{C_0(\widetilde V)}^{\Gamma}\otimes {\bf tr}$ in $KK_\R^\Gamma(C_0(\widetilde V),C_0(\widetilde V))$.
 
These equivariant bimodules being unitarily equivalent they represent the same class in $KK_\R^\Gamma(C_0(\widetilde V),C_0(\widetilde V))$.
\end{proof}
\end{prop}
Thus, for every finite dimensional unitary representation $\alpha :\Gamma\to U_n$, we get an element $\rho_\alpha ^{C_0(\widetilde V)}\in KK^{1,\Gamma}_{\R/\Z}(C_0(\widetilde V),C_0(\widetilde V))$.

We now explain the relation of this element with the element $[\alpha]_{\textrm{APS}}\in K^1_{\R/\Z}(C(V))$ constructed by Atiyah Patodi and Singer \cite{APS2, APS3}.

\medskip 

Let $Y$ be a locally compact space on which $\Gamma$ acts and $A$ and $B$ be $C_0(Y)-\Gamma$-algebras. Denote by $\mathcal{R}KK^\Gamma(Y;A,B)$ the Kasparov equivariant $KK$-theory group defined in \cite[\S 2]{Kas2}. In the language of $KK$-theory equivariant by groupoids introduced by Le Gall (\cite{LeGall}), $\mathcal{R}KK^\Gamma(Y;A,B)=KK^{\widetilde V\rtimes \Gamma}(A,B)$. 

If $\Gamma$ acts freely and properly on $Y$ there is a descent isomorphism (see \cite{Kas2}) $$\lambda^\Gamma:\mathcal{R}KK^\Gamma(Y;A,B)\longrightarrow \mathcal{R}(Y/\Gamma;A^\Gamma,B^\Gamma)$$ where $A^\Gamma$ and $B^\Gamma$ are the algebras of \emph{invariant elements}. In the language of \cite{LeGall}, this is the isomorphism corresponding to the Morita equivalent groupoids $Y\rtimes \Gamma$ and $Y/\Gamma$.

Let $X$ be a compact space, $A$ a $C^*$-algebra and $B$ be a $C(X)$-algebra. Then we have a canonical isomorphism $\mathcal{R}KK(X;A\otimes C(X),B)=KK(A,B)$. 

Let $Y$ be a free and proper and cocompact $\Gamma$-space. Put $X=Y/\Gamma$. Let $A$ be a $C^*$-algebra and let $B$ be a $C_0(Y) -\Gamma$-algebra. We have a sequence of isomorphisms
$$KK(A,B^\Gamma)\simeq \mathcal{R}KK(X;A\otimes C(X),B^\Gamma)\simeq\mathcal{R}KK^\Gamma(Y;A\otimes C_0(Y), B).$$
Using the forgetful map $\mathcal{R}KK^\Gamma(Y;A\otimes C_0(Y),B)\to  KK^\Gamma(A\otimes C_0(Y), B)$ we obtain a morphism $U_{A,B}^{Y,\Gamma}:KK(A,B^\Gamma)\to KK^\Gamma(A\otimes C_0(Y),B)$.

Using the Morita equivalence of $C_0(Y)\rtimes \Gamma$ with $C(X)$ and of $B\rtimes \Gamma$ with $B^\Gamma$ we also have a sequence of maps 
$$
 KK^\Gamma(A\otimes C_0(Y),B) \overset{j_\Gamma}\longrightarrow  KK(A\otimes C_0(Y)\rtimes \Gamma,B \rtimes \Gamma) 
\simeq  KK(A\otimes C(X),B^\Gamma)\overset{i^*_X}\longrightarrow KK(A,B^\Gamma)
$$
where $i_X:\C\to C(X)$ is the unital inclusion. We thus obtain a morphism: $$V_{A,B}^{Y,\Gamma}:KK^\Gamma(A\otimes C_0(Y),B)\to KK(A,B^\Gamma).$$

\begin{lem}\label{lemMishchAPS}
The composition  $V_{A,B}^{Y,\Gamma}\circ U_{A,B}^{Y,\Gamma}$ is the identity of $KK(A,B^\Gamma)$.
\end{lem}
\begin{proof}
Since $X=Y/\Gamma $ is compact, $B^\Gamma$ is the set of multipliers $b$ of $B$ which are $\Gamma$-invariant and satisfy $fb\in B$ for all $f\in C_0(Y)$. If $E$ is a $\Gamma$-equivariant Hilbert $B$ module, we may in the same way define the Hilbert $B^\Gamma$ module $E^\Gamma$, for instance by considering $E$ as a subspace of the $\Gamma$--$C_0(Y)$-algebra $\cK(E\oplus B)$ and then $E^\Gamma\subset \cK(E\oplus B)^\Gamma$.

We start with a pair $(E,F)$ representing an element in $\mathcal{R}KK^\Gamma(Y;A\otimes C_0(Y), B)$; by properness of the action (or, equivalently using the isomorphism $\mathcal{R}KK(X;A\otimes C(X),B^\Gamma)\simeq\mathcal{R}KK^\Gamma(Y;A\otimes C_0(Y), B)$), we may assume $F$ is $\Gamma $-invariant. It then defines an element $F_0\in \cL(E^\Gamma)$. The elements $(E^\Gamma,F_0)$ and $(E,F)$ correspond to each other via the isomorphism $\mathcal{R}KK(X;A\otimes C(X),B^\Gamma)\simeq\mathcal{R}KK^\Gamma(Y;A\otimes C_0(Y), B).$

The element corresponding to $(E^\Gamma,F_0)$ via the isomorphism $\mathcal{R}KK(X;A\otimes C(X),B^\Gamma)\simeq KK(A,B^\Gamma)$ is of course $(E^\Gamma,F_0)$ where we just do not consider the left $C(X)$ action !

The element corresponding to $(E,F)$ under the forgetful morphism $\mathcal{R}KK^\Gamma(Y;A\otimes C_0(Y),B)\to  KK^\Gamma(A\otimes C_0(Y), B)$ is of course $(E,F)$; its image by $j_\Gamma$ is $(E\rtimes \Gamma,F\otimes 1)$ where $E\rtimes \Gamma=E\otimes_B(B\rtimes \Gamma)$. The Hilbert  $(C(X)\otimes A,B^\Gamma)$ bimodule corresponding to it under the Morita equivalences, is then seen to be $E^\Gamma$ - and the Lemma follows. 
\end{proof}


According to Remark \ref{remequ}, we define $\mathcal{R}KK^\Gamma_{\mathbb{R}}(V;A,B)$ and $\mathcal{R}KK^\Gamma_{\R/\mathbb{Z}}(V;A,B)$ using inductive limits over ${\rm II}_1$-factors.

The APS element in $K^1_{\R/\Z}(V)$ was described in \cite{AAS} as an element $[\alpha]_{\textrm{APS}}\in K_0(Z_{i,i}\otimes C(V))$.

\begin{prop}\label{mischAPS} 
The element $\rho_\alpha^{C_0(\widetilde V)}$ is the image of $[\alpha]_{\textrm{APS}}$ under the composition of the isomorphisms $$K_{1,\R/\Z}(C(V))\simeq \mathcal{R}KK_{\R/\Z}^1(V;C(V),C(V))\simeq \mathcal{R}KK_{\R/\Z}^{1,\Gamma}(\widetilde V;C_0(\widetilde V),C_0(\widetilde V)\otimes Z_{i,i})$$ with the forgetful map
$$\mathcal{R}KK_{\R/\Z}^{1,\Gamma}(\widetilde V;C_0(\widetilde V),C_0(\widetilde V))\to KK_{\R/\Z}^{1,\Gamma}(C_0(\widetilde V),C_0(\widetilde V)).$$
\end{prop}
\begin{proof}
In \cite[Prop. 5.1 and Prop. 5.2]{AAS}, we represented the class $[\alpha]_\textrm{APS}$ as the image under the morphism $K_0(C(V)\otimes Z_{i,i})\to K_{1,\R/\Z}(C(V))$ of the $C(V)\otimes Z_{i,i}$ module
 $\mathcal{E}=\{f\in C([0,1], E^+\otimes N); f(0)\in E^+\otimes 1 \;,\;w_v f(1) \in E^-\otimes 1\}$ where $N$ is a $\rm{II}_1$-factor with tracial inclusion $\lambda:\Gamma\to N$,  $i:\C\rightarrow N$ is the unital inclusion, $E^+$ is the flat vector bundle associated with $\alpha$ and $E^-$ is the trivial vector bundle with rank equal to the dimension of $\alpha$. The unitary $w_v$ is given by $w_v:=(\phi\otimes 1_{E^-})^{-1}\circ v^{-1}\circ (\phi\otimes 1_{E^+})$ where $\phi:C(V)\otimes N\to \mathcal {E}$ is an isomorphism, \emph{i.e.} a trivialization of the flat bundle $\mathcal{E}$ associated with $\lambda$ and $v$ is an isomorphism of flat bundles $\mathcal{E}\otimes E^- \to \mathcal{E}\otimes E^+$.

 At the level of $\widetilde V$, the trivialization $\phi$ corresponds to a $\Gamma$-equivariant isomorphism $\tilde \phi:\widetilde V\otimes N\to \widetilde V\otimes P_\lambda$ and $v$ to an unitary $u\in N\otimes M_n(\C)$ intertwining $\lambda\otimes 1$ and $\lambda \otimes \alpha$ ($P_\lambda$ is as in section \ref{constrrhoalphaA}).
 
 Let then $\widehat {\mathcal{E}}$ denote the $C_0(\widetilde V)\otimes Z_i$ module $\{f\in C_0(\widetilde V\times [0,1])\otimes P_\lambda  f(0)\in \widetilde\phi(C_0(\widetilde V)\otimes 1)\}$. The pair $(\widehat {\mathcal{E}},0)$ defines a Kasparov bimodule in ${\bf E}^\Gamma(C_0(\widetilde V),C_0(\widetilde V)\otimes Z_{i})$ joining  $(C_0(\widetilde V),0)$ and $(C_0(\widetilde V)\otimes P_\lambda,0)$.
 
 We may then perform the ``$\diamond $'' construction as in Theorem \ref{Strongrhoconstruction} using the isomorphism $u$ to obtain an element in $z\in KK^\Gamma(C_0(\widetilde V),C_0(\widetilde V)\otimes Z_{i,i})$ whose image in $KK_{\R/\Z}^{1,\Gamma}(C_0(\widetilde V),C_0(\widetilde V))$ is $\rho_\alpha^{C_0(\widetilde V)}$. The element $z$ is of course given by the Hilbert $C_0(\widetilde V)\otimes Z_{i,i}$-module $$ \widetilde{\mathcal{E}}:=\{f:[0,1]\times \widetilde{V}\rightarrow \C^n\otimes N: f(0,x)\in {\C}^n\times 1_N, \, \widetilde{w_v}(x)f(1,x)\in  {\C}^n\times 1_N, \textrm{ for every } x\in \widetilde{V}\,,$$
 where $ \widetilde{w_v}=(\tilde\phi\otimes 1)^{-1}\circ u^{-1}\circ (\tilde\phi\otimes 1)$. As the left and right $C_0(\widetilde V)$ actions on $ \widetilde{\mathcal{E}}$ coincide, the element $z$ is in the image of the forgetful map, $\mathcal{R}KK^\Gamma(\widetilde V;C_0(\widetilde V),C_0(\widetilde V)\otimes Z_{i,i})\rightarrow KK^\Gamma(C_0(\widetilde V),C_0(\widetilde V)\otimes Z_{i,i})$. Now $\widetilde{\mathcal E}^\Gamma=\mathcal E$, and the result follows.
 \end{proof}

It follows then from Proposition \ref{mischAPS} and Lemma \ref{lemMishchAPS} that we may deduce the class $[\alpha]_{APS}$ from our class $\rho_{\alpha}^{C_0({\widetilde V})}$.

 \subsection{If $A$ is proper and free}
 
Out of the example of the Mishchenko bundle of a compact manifold, we then obtain the following statements. Let $Z$ be a locally finite $CW$ complex realization for the classifying space $B\Gamma$ and $q:\widetilde Z\to Z$ the corresponding covering with group $\Gamma$. Write $Z=\bigcup Z_n$, where $Z_n$ is an increasing sequence of finite subcomplexes. Denote by $N$ a ${\rm II}_1$-factor containig $\Gamma$ in a trace preserving way, and $E_N=\widetilde Z\times _\Gamma N$ the corresponding bundle over $Z$ with fiber $N$.

\begin{lem}\label{prop4.3}
\begin{enumerate}
\item The restriction to $Z_n$ of the bundle $E_N$ is equal in $K^0_\R(Z_n)$ to the unit element (corresponding to the trivial bundle).
\item If $\widetilde X$ is a free, proper and cocompact $\Gamma$-space, then there is a ${\rm II}_1$-factor $N$ with an embedding $\lambda:C^*\Gamma\longrightarrow N$ such that ${\rm tr}_N\circ \lambda={\bf tr}$ and a unitary $w\in \cL(C_0(\widetilde X)\otimes N)$ intertwining the $\Gamma$-actions $\beta\otimes \id_N$ and $\beta\otimes \lambda$ on $C_0(\widetilde X)\otimes N$ where $\beta$ is the action of $\Gamma$ on $C_0(\widetilde X)$ by translation.\label{prop4.3.2}
\end{enumerate}
\begin{proof}
\begin{enumerate}
\item Since $Z_n$ is a finite CW-complex, its $K$-theory is a finitely generated group. Therefore, using the Rosenberg-Schochet universal coefficient formula (\cite{RoSc, Ros}), we find 
$KK_\R(C_0(Z_n), \C)={\rm Hom}(KK(\C,C_0(Z_n));\R)$.
 Using the Baum-Douglas theory (\textit{cf.} \cite{BaDo}), we know that the $K$-homology of $Z_n$ is generated by cycles given by elements $g_*(x)$ where $g$ is a continuous map  $g:V\to Z_n$, $V$ is a compact manifold and $x\in K_*(V)$ is an element in the $K$-homology of $V$. But by Atiyah's theorem for covering spaces (\cite{Ati}), the pairing of $E_N$ with such a $K$-homology element coincides with the index. It follows that the class of $E_N$ is $1$. 

\item We have a continuous classifying map $f:X\to B\Gamma$ where $X=\widetilde X$. Since $X$ is compact its image sits in some $Z_n$ and it follows that the bundle $f^*(E_N)$ defines the trivial element in $K^0_\R(X)$. Up to replacing $N$ by $M_k(N)$, we may then assume that the bundle $f^*E_N$ is trivial (as explained in \cite[Prop. 5.1]{AAS}).  By definition of $E_N$, this means exactly that there exists a unitary $w\in \cL(C_0(\widetilde X)\otimes N)$ intertwining the $\Gamma$-actions $\beta\otimes \id_N$ and $\beta\otimes \lambda$ on $C_0(\widetilde X)\otimes N$.\qedhere
\end{enumerate}
\end{proof}
\end{lem}

We now recall the definition of free and proper $\Gamma$-algebras. Recall (\cite{Kas2}) that, if $X$ is a locally compact space, a $C_0(X)$ algebra is a $C^*$-algebra $A$ endowed with a morphism from $C_0(X)$ into the center of the multiplier algebra of $A$ and such that $A=C_0(X)A$. If $X$ is a $\Gamma$ space, a $C_0(X)-\Gamma$-algebra $A$ is an algebra endowed with compatible structures of $C_0(X)$-algebra and  $\Gamma$-algebra (\textit{i.e.} the morphism $C_0(X)\to \mathcal{ZM}(A)$ is equivariant).

\medskip Let us recall facts about $C_0(X)$-algebras (see \cite{Kas2,LeGall}):

\begin{properties}
\begin{enumerate}
\item \label{prop.4.4.1}If $A$ is a $C_0(X)$-algebra, we may define for every open subset $U\subset X$ the $C_0(U)$-algebra $A_U=C_0(U)A$. For every closed subset $F\subset X$, we put $A_F=A/A_{F^c}$, which is a $C_0(F)$ algebra. In particular, we have a fiber $A_x$ for every point $x\in X$. Moreover, there is a natural evaluation map $a\mapsto a_x$ from $A$ to $A_x$ and for $a\in A$, we have $\|a\|=\sup \{\|a_x\|,\ x\in X\}$.
\item If $A$ is a $C_0(X)$-algebra and $f:Y\to X$ is a continuous map, we define a pull back $f^*(A)$ which is a $C_0(Y)$-algebra. This pull back satisfies $f^*(A)_y=A_{f(y)}$ for all $y\in Y$.\\ We will use the following facts:\begin{enumerate}
\item  If $T$ is a locally compact space and $p:X\times T\to X$ is the projection then $p^*(A)=A\otimes C_0(T)$. 
\item If $h:Y\to Y$ is a homeomorphism such that $f\circ h=f$, there is an automorphism $\theta^h:f^*(A)\to f^*(A)$ such that $(\theta^h(b))_y=b_{h(y)}$ for every $y\in Y$. Note that by property (\ref{prop.4.4.1}) this equality characterizes $\theta^h(b)$.
\end{enumerate}
\end{enumerate}
\end{properties}

\begin{dfn}
A $\Gamma$-algebra $A$ is said to be \emph{free} (\textit{resp.} \emph{proper}) if there exists a free (\textit{resp.} proper) $\Gamma$-space $\widetilde X$ such that $A$ is a $C_0(\widetilde X) - \Gamma $-algebra.
\end{dfn}

Note that if $A$ is free and proper, then there exist a free $\Gamma$-space $\widetilde X_1$ and a proper $\Gamma$-space $\widetilde X_2$ such that $A$ is a $C_0(\widetilde X_1) - \Gamma $-algebra and a $C_0(\widetilde X_2)$--$\Gamma $-algebra. Then $A$ is a $C_0(\widetilde X_1\times \widetilde X_2) - \Gamma $-algebra, and $\widetilde X_1\times \widetilde X_2$ is free and proper.

\medskip
We will next prove that every free and proper $\Gamma$-algebra is $K$-theoretically free and proper (Theorem \ref{SKFPth}).

Let us start with the cocompact case:

\begin{prop}\label{cocompactAlg}
If $\widetilde X$ is a free, proper and cocompact $\Gamma$-space, then every $C_0(\widetilde X)-\Gamma$-algebra satisfies \SKFP.
\begin{proof}
Let $A$ be a $C_0(\widetilde X) - \Gamma $-algebra. Extending the action by left multiplication of $C_0(\widetilde X)\otimes N$ on $A\otimes N$ to the multiplier algebra, we find a unitary $\hat w\in \cL(A\otimes N)$ intertwining the actions $\hat \beta\otimes\id_N$ and $\hat \beta\otimes \lambda$ of $\Gamma$. It follows that $A$ satisfies \SKFP.
\end{proof}\end{prop}

\medskip 
As above, let $Z$ be a locally finite $CW$ complex realization for the classifying space $B\Gamma$. Write $Z=\bigcup Z_n$, where $Z_n$ is an increasing sequence of finite subcomplexes. Using local finiteness, we will also assume that $Z_n$ is contained in the interior of $Z_{n+1}$. 

Let $\widetilde X$ be a a free and proper $\Gamma$-space and let $A$ be a (separable) $C_0(\widetilde X) - \Gamma $-algebra. Put $X=\widetilde X/\Gamma$ and let $q:\widetilde X\to X$ be the quotient. Let also $f:X\to Z$ be a classifying map. 

Define $\widetilde X_n=(f\circ q)^{-1}(\mathaccent 023{Z}_n)\subset \widetilde X$ and   $\Omega =\bigcup _{n\in \mathbb{N}} \widetilde X_n\times ]n,+\infty[ $.

Let $\cA=C_0(\Omega )(A\otimes C_0(\R))$ and $A_n=C_0(X_n)A$. Let $\ccA$ denote the $c_0$-sum $\bigoplus_{n\in \mathbb{N}} A_n$ and $j:\ccA\to\ccA$ the shift $(a_0,a_1,a_2,\ldots)\mapsto (0,a_0,a_1,a_2,\ldots)$.

\begin{lem}
\begin{enumerate}
\item The $\Gamma$-algebras $A\otimes C_0(\R)$ and $\cA$ are homotopy equivalent.
\item The algebra $\cA$ identifies with the torus algebra $\cT(j,\id_{\ccA})$ (see definition \ref{torus}).
\end{enumerate}
\begin{proof}
\begin{enumerate}
\item Of course $A\otimes C_0(\R)$ and $A\otimes C_0(\R_+^*)$ are homotopy equivalent. Moreover let $h:Z\to \R_+$ be a (proper) map such that $h(z)>n$ if $z\not\in Z_n$, and put $\hat h=h\circ f\circ q:X\to \R_+^*$. The map $(x,t)\mapsto (x,t+\hat h(x))$ is a homeomorphism  $\widetilde X\times \R\to \widetilde X\times \R$. It induces an equivariant automorphism of $A\otimes C_0(\R)$ mapping $A\otimes C_0(\R_+^*)$ into $\cA$ which is a homotopy inverse of the inclusion $\cA\to A\otimes C_0(\R_+^*)$.
\item An element in $A\otimes C_0(\R_+^*)$ is a map $\xi:\R_+^*\to A$. Such a $\xi $ is in $\cA$ if and only if, for all $n\in \N$ and $t\in ]n,n+1]$, we have $\xi(t)\in A_n$ for $n<t\le n+1$. Associated to $\xi $ is a map $\zeta:[0,1]\to \ccA$ defined by $\zeta(t)_n=\xi(n+t) \ (\in A_n$). It is immediately seen that $\zeta \in \cT(j,\id_{\ccA})$ and that $\xi\mapsto \zeta$ is an isomorphism. \qedhere
\end{enumerate}
\end{proof}
\end{lem}

\begin{lem}
\begin{enumerate}
\item For every $n$, the algebra $A_n$ satisfies \SKFP.
\item The algebra $\ccA$ satisfies \SKFP.
\item The algebra $\cA$ satisfies \SKFP.
\item The algebra $A$ satisfies \SKFP.
\end{enumerate}
\begin{proof}
\begin{enumerate}
\item Through the equivariant map $\widetilde X_n\to \widetilde Z_n$, the algebra $A_n$ is a $\widetilde Z_n$ algebra. As $\widetilde Z_n$ is a free proper and cocompact $\Gamma$-space, the algebra $A_n$ satisfies \SKFP\ by prop.\ref{cocompactAlg}.
\item For every $\Gamma$-algebra $B$ the map $KK_\R^\Gamma(\ccA,B)\prod _n KK_\R^\Gamma(A_n,B)$ is an isomorphism - the inverse map associates to a sequence $(E_n,F_n)$ of bimodules (with $F_n$ bounded) defining elements in $KK_\R^\Gamma(A_n,B)$ the class of $(\bigoplus E_n,F)$ where $F$ is defined by $F((x_n)_n)=(F_n(x_n))_n$ (for $(x_n)\in \bigoplus _nE_n$) (see \cite{Ros, RoSc}). Now $1_\R^\Gamma-[{\bf tr}]^\Gamma$ acts trivially on each copy $KK_\R^\Gamma(A_n,B)$, whence $1_\R^\Gamma-[{\bf tr}]^\Gamma$ acts as the $0$ element in $KK_\R^\Gamma(\ccA,\ccA)$.
\item Write $\cA=\cT(j,\id_{\ccA})$ and let $p:\cA\to \ccA$ and $h:S\ccA\to \cA$ be the associated map maps ($p:(a,f)\mapsto a$ and $h:f\mapsto (0,f)$). We have $p_*(1_{\cA}^\Gamma\otimes _\C(1_\R^\Gamma-[{\bf tr}]^\Gamma))=p^*(1_{\ccA}^\Gamma\otimes _\C(1_\R^\Gamma-[{\bf tr}]^\Gamma))=0$ since  $\ccA$ satisfies \SKFP.

By the torus exact sequence (prop. \ref{torexS}), there exists $y\in KK_\R^\Gamma(\ccA,S\ccA)$ such that $1_{\cA}^\Gamma\otimes _\C(1_\R^\Gamma-[{\bf tr}]^\Gamma)=h_*y$. As $\ccA$ satisfies \SKFP, we have $y\otimes _{S\ccA}(1_{S\ccA}^\Gamma\otimes _\C(1_\R^\Gamma-[{\bf tr}]^\Gamma))=0$. As the Kasparov product over $\C$ is commutative, we find $$y\otimes _{S\ccA}(1_{S\ccA}^\Gamma\otimes _\C(1_\R^\Gamma-[{\bf tr}]^\Gamma))=(1_{\cA}^\Gamma\otimes _\C(1_\R^\Gamma-[{\bf tr}]^\Gamma))\otimes _{\cA}y$$ whence $(1_{\cA}^\Gamma\otimes _\C(1_\R^\Gamma-[{\bf tr}]^\Gamma))\otimes _{\cA}h_*y$, which yields $$1_{\cA}^\Gamma\otimes _\C\Big((1_\R^\Gamma-[{\bf tr}]^\Gamma)\otimes _\C(1_\R^\Gamma-[{\bf tr}]^\Gamma)\Big)=(1_{\cA}^\Gamma\otimes _\C(1_\R^\Gamma-[{\bf tr}]^\Gamma))\otimes _{\cA}(1_{\cA}^\Gamma\otimes _\C(1_\R^\Gamma-[{\bf tr}]^\Gamma))=0.$$ But $[{\bf tr}]^\Gamma$ is idempotent in $KK_\R^\Gamma(\C,\C)$, thus $(1_\R^\Gamma-[{\bf tr}]^\Gamma)^2=(1_\R^\Gamma-[{\bf tr}]^\Gamma)$.

\item The $\Gamma$-algebra $SA$ is homotopy equivalent to $\cA$. It follows that $SA$ satisfies \SKFP; thus $1_{SA}^\Gamma\otimes _\C(1_\R^\Gamma-[{\bf tr}]^\Gamma)=0$. The Bott isomorphism $x\mapsto Sx$ from $KK_\R^\Gamma(C,D)$ to $KK_\R^\Gamma (SC,SD)$ maps $1_{A}^\Gamma\otimes _\C(1_\R^\Gamma-[{\bf tr}]^\Gamma)$ to $1_{SA}^\Gamma\otimes _\C(1_\R^\Gamma-[{\bf tr}]^\Gamma)=0$. It follows that $A$ satisfies \SKFP. \qedhere
\end{enumerate}
\end{proof}
\end{lem}

\medskip We have proved:

\begin{theorem}\label{SKFPth}
Every  free and proper $\Gamma $-algebra  satisfies \SKFP. \hfill$\square$
\end{theorem}
 
 As a corollary, every $\Gamma $-algebra which is $KK$-subequivalent to a proper and free algebra satisfies \SKFP.
 
 \subsection{If $\Gamma$  satisfies (the $KK^\Gamma$- form of) the Baum-Connes conjecture}
 
Let us say that $\Gamma$ satisfies the $KK^\Gamma$-form of the Baum-Connes conjecture if there is a proper $\Gamma$-algebra $\mathcal {Q}$ such that $\C$ is $KK^\Gamma$-subequivalent to $\cQ$. This of course implies that $\Gamma$ is $K$-amenable (\cite{Cuntz}). On the other hand, Higson and Kasparov proved that every A-$T$-menable group satisfies this property (\cite{HiKa1, HiKa2}).

If $\Gamma$ satisfies the $KK^\Gamma$-form of the Baum-Connes conjecture every $\Gamma$-algebra  is $KK^\Gamma$-subequivalent to a proper $\Gamma$-algebra (namely $A$  is $KK^\Gamma$-subequivalent to $\cQ\otimes A$).

\begin{itemize}
\item If $\Gamma$ is torsion free. Then $\cQ$ is automatically free. Then every $\Gamma$-algebra satisfies \SKFP.

\item In the general case. Assume $A$ is free (\footnote{Note that we only need to assume that $\cQ\otimes A$ is free, which is much weaker: it is just a statement on the torsion elements of $\Gamma$.}), then  $\cQ\otimes A$ is free and proper and therefore satisfies \SKFP. It follows that every free $\Gamma $-algebra satisfies \SKFP. 
\end{itemize}

\subsection{If $\Gamma$ has a $\gamma$ element}
 
 Recall (\emph{cf.} \cite{Kas2, Tu}) that a $\gamma$-element for $\Gamma$ is an element $\gamma\in KK^\Gamma(\C,\C)$ such that:
 \begin{enumerate}
\item There exists a proper $\Gamma $-algebra $\cQ$ and elements $D\in KK_\Gamma(\cQ,\C)$ and $\eta \in KK_\Gamma(\C,\cQ)$ such that $\gamma=\eta\otimes _\cQ D$. $D$ and $\eta$ are respectively the so called \emph{Dirac} and \emph{dual Dirac} element;
\item for every proper $\Gamma$-algebra $A$, $1_A^\Gamma\otimes _\C\gamma=1_A^\Gamma$.
\end{enumerate}

Recall (\emph{cf.} \cite{Tu}), that if $\gamma$ exists, it is unique.

\begin{itemize}
\item If $\Gamma$ is torsion free and has a $\gamma$ element, then  the algebra $\cQ$ is free and proper, therefore it satisfies \SKFP. It follows that, for every finite dimensional unitary representation $\alpha :\Gamma\to U_n$, we can construct a canonical element $\rho_\alpha ^{\cQ}\in KK_{\R/\Z}^{1,\Gamma}(\cQ ,\cQ)$ and use the element $\eta  $ and $D $ in order to define a canonical element $\rho_\alpha ^{\C}=\eta\otimes _{\cQ} \rho_\alpha ^{\cQ}\otimes_{\cQ}D \in KK_{\R/\Z}^{1,\Gamma}(\C,\C)$ in ``the image of $\gamma$''. Then for every $\Gamma $ algebra, we construct $\rho_\alpha^A=1_A^\Gamma\otimes_\C \rho_\alpha ^{\C}$.

\item In the same way as above, if $\Gamma$ is no longer assumed to be torsion free, we may construct the element $\rho_\alpha ^{\cQ\otimes A}$ if $A$ is free - and more generally if $\cQ\otimes A$ is free. We then use the Dirac and dual Dirac elements to construct a canonical element $\rho_\alpha ^{A}=\eta\otimes _{\cQ} \rho_\alpha ^{\cQ\otimes A}\otimes_{\cQ}D\in KK_{\R/\Z}^{1,\Gamma}(A,A)$ in ``the image of $\gamma$''.
\end{itemize}
 
 \subsection{If $\alpha(\Gamma)$ is torsion free}
 
 Put $\check\Gamma =\alpha(\Gamma)$ (with the discrete topology). Since $\check\Gamma$ is linear, it follows from \cite{Weinberger et al} that it has a $\gamma$ element. We may thus define an element $\rho_{\check\alpha} ^{\C}\in KK_{\R/\Z}^{1,\check\Gamma}(\C,\C)$ where $\check\alpha$ is the inclusion $\check\Gamma\subset U_n$.
 Then we can use the morphism $q:\Gamma\to \check \Gamma$ in order to define $\rho_{\alpha} ^{\C}=q^*\rho_{\check\alpha} ^{\C}$ (and then put, for every $\Gamma $-algebra $\rho_\alpha^A=1_A^\Gamma\otimes_\C \rho_\alpha ^{\C}$).

\section{Weak (KFP) property}
\label{rho}

\subsection{KK-theory elements with real coefficients associated with a trace on $C^*\Gamma$}

As explained in Lemma \ref{lemma1}, for every $\Gamma$-algebra $A$ and every $C^*$-algebra $B$, we may identify $KK(A\rtimes \Gamma,B)$ with $KK^\Gamma(A,B)$ where $B$ is endowed with the trivial action of $B$. Replacing in this formula $B$ by $B\otimes M$ where $M$ is a ${\rm II}_1$-factor and taking the inductive limit, we find for a $\Gamma$-algebra $A$ and a $C^*$-algebra $B$ endowed with a trivial $\Gamma$-action an identification $KK_{\R}^\Gamma(A,B)=KK_{\R}(A\rtimes \Gamma,B)$ - and in the same way, $KK_{\R/\Z}^\Gamma(A,B)=KK_{\R/\Z}(A\rtimes \Gamma,B)$.

\begin{dfn}
Given a $\Gamma$-algebra $A$ and a trace $\tau$ on $C^*\Gamma$ we have three equivalent ways to define the same element $[\morph{A}{\tau}]\in KK_{\R}(A\rtimes \Gamma,A\rtimes \Gamma)$:
\begin{enumerate}
\item We may put $[\morph{A}{\tau}]=\tau_*[\delta^A]$ where $\delta^A\in KK(A\rtimes \Gamma,A\rtimes \Gamma\otimes C^*\Gamma)$ is the class of the morphism $\delta^A$ and $\tau_*:KK(B,C\otimes C^*\Gamma)\longrightarrow KK_{\R}(B,C)$ is the map associated with the trace $\tau$ constructed in Definition \ref{tau*}.
\item Starting with $[\tau]\in KK_{\R}(C^*\Gamma,\C)$, we obtain (by tensoring with $A\rtimes \Gamma$) an element $(1_{A\rtimes \Gamma}\otimes [\tau])\in KK_{\R}(A\rtimes \Gamma\otimes C^*\Gamma,A\rtimes \Gamma)$ and, with this notation, we have $[\morph{A}{\tau}]=(\delta^A)^*(1_{A\rtimes \Gamma}\otimes [\tau])\in KK_{\R}(A\rtimes \Gamma,A\rtimes \Gamma)$.
\item Let as usual $[\tau]^\Gamma\in KK_\R^\Gamma(\C,\C)$ be the element corresponding to $[\tau]$ via the identification of
 $KK_{\R}(C^*\Gamma,\C)$ with $KK^\Gamma_{{\R}}(\C,\C)$; the element $1_A^\Gamma\otimes [\tau]^\Gamma\in KK^{\Gamma}_{\R}(A,A)$ satisfies $[\morph{A}{\tau}]=j_\Gamma(1_A^\Gamma\otimes [\tau]^\Gamma)$ where $j_\Gamma:KK^{\Gamma}_{\R}(A,A)\to KK_{\R}(A\rtimes \Gamma,A\rtimes \Gamma)$ is Kasparov's descent.
\end{enumerate}
\end{dfn}

\subsection{Weakly K-theoretically ``free and proper'' algebras}

We now introduce a weakening of (KFP) property

\begin{dfn}
Let $\Gamma$ be a discrete group. Denote by ${\bf tr}$ its group tracial state. We say that the $\Gamma$-algebra $A$ satisfies \KFP\ if $[\morph{A}{{\bf tr}}]$ is equal to the unit element $1_{A\rtimes \Gamma,\R}$ of the ring $KK_{\R}(A\rtimes \Gamma,A\rtimes \Gamma)$.
\end{dfn}

This conditions means that there is a $\rm{II}_1$-factor $N$ with a morphism $\lambda:C^*\Gamma\longrightarrow N$ such that ${\rm tr}_N\circ \lambda={\bf tr}$ (where ${\rm tr}_N$ is the normalized trace of $N$) and such that the maps $\morph{A}{\lambda}=(\id_{A\rtimes \Gamma}\otimes \lambda)\circ \delta^A: A\rtimes \Gamma \longrightarrow (A\rtimes \Gamma)\otimes N$ and $\iota^A= \operatorname{id}\otimes 1_N:A\rtimes \Gamma \longrightarrow (A\rtimes \Gamma)\otimes N$ define the same element in $KK(A\rtimes \Gamma,A\rtimes \Gamma\otimes N)$.

Let us make a few comments on this definition:

\begin{properties}\begin{enumerate}
\item If a $\Gamma$-algebra $A$ satisfies \SKFP, it satisfies \KFP. \\ 
Indeed, if $1_A^\Gamma\otimes [{\bf tr}]^\Gamma=1_{A,\R}^\Gamma$, then $[\morph{A}{{\bf tr}}]=j_\Gamma(1_A^\Gamma\otimes [{\bf tr}]^\Gamma)=j_\Gamma(1_{A,\R}^\Gamma)=1_{A\rtimes \Gamma,\R}$.
\item If $\C$ satisfies \KFP, it satisfies \SKFP. Indeed, $\varepsilon_* [\morph{A}{{\bf tr}}]=[{\bf tr}]\in KK_\R(C^*\Gamma,\C)=KK^\Gamma_\R(\C,\C)$.

\item Let $A,B$ be $\Gamma$-algebras. If $A$ is $KK^\Gamma$-subequivalent to $B$ and $B$ satisfies \KFP\  then so does $A$. \\
 The assumption means that there exist $x\in KK^\Gamma(A,B)$ and $y\in KK^\Gamma(B,A)$ satisfying $x\otimes_By=1_A^\Gamma$. Note that $y\otimes _\C[{\bf tr}]^\Gamma=[{\bf tr}]^\Gamma\otimes _\C y=(1_B^\Gamma\otimes [{\bf tr}]^\Gamma)\otimes _B y$.
Assume $B$ satisfies \KFP. Then $j_\Gamma(y\otimes_\C [{\bf tr}]^\Gamma)$  is the image in $KK_{\R} (B\rtimes \Gamma,A\rtimes \Gamma)$ of $j_\Gamma(y)$, whence $[\morph{A}{{\bf tr}}]=j_\Gamma(x)\otimes _{B\rtimes \Gamma}j_\Gamma(y\otimes_\C [{\bf tr}]^\Gamma)=1_{A\rtimes \Gamma,\R}$.
\end{enumerate}
\end{properties}

\begin{rems} \label{prodtraces}
\begin{enumerate}
\item (see Remark \ref{strongprodtraces}) If $A$ is satisfies \KFP, then  $[\morph{A}{\tau}]=1_{A\rtimes \Gamma,\R}$ for any tracial state $\tau$. 

\item (see Remark \ref{qGammainvertible}) Given a von Neumann algebra $N$ and a morphism $\lambda:C^*\Gamma\longrightarrow N$ such that ${\rm tr}_N\circ \lambda={\bf tr}$ (where ${\rm tr}_N$ is the normalized trace of $N$), then $\lambda$ factors through $C^*_{red}(\Gamma)$; it therefore defines a morphism  $\widehat{\morph{A}{\lambda}}:A\rtimes_{red} \Gamma\longrightarrow A\rtimes \Gamma\otimes N$. 
Let $q_\Gamma:A\rtimes \Gamma\to A\rtimes_{red}\Gamma$ be the natural morphism.
If $A$ satisfies \KFP, then $q_\Gamma^*[\widehat{\morph{A}{\lambda}}]= 1_{A\rtimes\Gamma,\R}$ \textit{i.e} the morphism $q_\Gamma:A\rtimes \Gamma\to A\rtimes_{red}\Gamma$ has a $KK_{\R}$ one sided inverse. 
\end{enumerate}
\end{rems}

\subsection{A construction involving coactions and torus algebras}
Throughout this section, the following objects will be fixed: $\Gamma$ is a discrete group, $A$ is a $\Gamma$-algebra, $\alpha$ is 
 a finite dimensional unitary representation of $\Gamma$. 
 Moreover, we fix a ${\rm II}_1$-factor $N$ and a morphism $\lambda:C^*\Gamma\longrightarrow N$ such that ${\rm tr}_N\circ \lambda={\bf tr}$, where as above ${\rm tr}_N$ denotes the normalized trace of $N$ and ${\bf tr}$ is the group trace of $\Gamma$.
 
 We fix two morphisms $\iota^A,\morph{A}{\lambda}:A\rtimes \Gamma\to A\rtimes \Gamma\otimes N$ where $\iota ^A:x\mapsto x\otimes 1$ and $\morph{A}{\lambda}$ uses the ``coaction'' of $\Gamma$ and is defined in Notation \ref{notation} below.

We now construct an element  $\sigma^A_\alpha \in KK^1_{\R/\Z}(\cT,A\rtimes \Gamma)$ where $\cT=\mathcal{T}({\iota^A,\morph{A}{\lambda}})$ is the corresponding torus algebra. This will be an important ingredient in our construction of the weaker $\rho$ invariant (Definition \ref{defweakrho}).

\subsubsection{A morphism between torus algebras}

There exists a $\rm{II}_1$-factor $M$ containing $N$ with a morphism $C(U_n)\rtimes \Gamma\longrightarrow M$ extending $\lambda :C^*\Gamma\longrightarrow N\subset M$. Indeed we can take as $M$ (a $\rm{II}_1$-factor containing) the  free product of $N$ with $L^{\infty}(U_n)\rtimes \Gamma$ with amalgamation over the group von Neumann algebra of $\Gamma$.

It follows, thanks to Lemma \ref{strongunitary}, that there exists a unitary $u\in M_n(\C)\otimes M$ satisfying $(\alpha(g)\otimes \lambda (g))u=u(1\otimes \lambda (g))\in M_n(\C)\otimes M$.

\bigskip

\begin{notation}\label{notation}
Given a unital $C^*$-algebra $B$ and a unitary representation $\pi:\Gamma\to B$ - in other words a morphism $\pi:C^*\Gamma \longrightarrow B$ we will put  $$\morph{A}{\pi}:=(\operatorname{id}_{A\rtimes \Gamma}\otimes \pi)\circ \delta^A:A\rtimes \Gamma\longrightarrow (A\rtimes \Gamma)\otimes B.$$ 
\end{notation}

We denote  by $\alpha:C^*\Gamma\to M_n(\C)$ the morphism corresponding to $\alpha:\Gamma\to U_n$. We obtain a morphism $\morph{A}{\alpha} =(\operatorname{id}_{A\rtimes \Gamma}\otimes \alpha)\circ \delta^A:A\rtimes \Gamma \longrightarrow (A\rtimes \Gamma) \otimes M_n$. Denote by $[\morph{A}{\alpha}]\in KK(A\rtimes \Gamma,A\rtimes \Gamma)$ the associated $KK$-class.

\begin{prop}\label{prop2.3} 

Let $M$ be a $\rm{II}_1$-factor containing $N$ and let $u\in M_n(\C)\otimes M$  be a unitary satisfying $(\alpha(g)\otimes \lambda (g))u=u(1\otimes \lambda (g))\in M_n(\C)\otimes M$.

Denote by $j^A_n:A\rtimes \Gamma \to A\rtimes \Gamma\otimes M_n(\C)$ the morphism given by $x\mapsto x\otimes 1_{M_n(\C)}$ and $\iota^A:A\rtimes \Gamma \longrightarrow A\rtimes \Gamma\otimes  N$ the one given by $x\mapsto x\otimes 1_N$. Let $i_n:M_n(\C)\rightarrow M_n(\C)\otimes M$ the map defined by $i_n(x)=x \otimes 1_M$ and $\widehat{\,\,i_n}={\rm Ad}_{u}\circ i_n$.

\begin{enumerate}
\item There is a morphism 
$\Delta_\alpha^u:\mathcal{T}(\iota^A,\delta^A_\lambda)\longrightarrow  A\rtimes \Gamma \otimes Z(i_n,\widehat{\,\,i_n})=Z(\operatorname{id}_{A\rtimes \Gamma} \otimes i_n, {\operatorname{id}}_{A\rtimes \Gamma} \otimes \widehat{\,\,i_n})$ such that $\Delta_\alpha^u(a,f)=(\delta^A_{\alpha}(a),j^A_n(a),g)$ with
$$
g(t)=\begin{cases}
(\morph{A}{\alpha}\otimes i_{M,N})(f(2t))&$for $t\le 1/2\\
(1_{A\rtimes \Gamma}\otimes u)(j^A_n\otimes i_{M,N})f(2-2t)(1_{A\rtimes \Gamma}\otimes u^*)& $for $t\ge 1/2.
\end{cases}
$$
\item Let $p:\mathcal{T}({\iota^A,\morph{A}{\lambda}})\longrightarrow A\rtimes \Gamma$ the projection $p:(a,f)\mapsto a$. With respect to the natural evaluation maps $p_0,p_1: A\rtimes \Gamma\otimes Z(i_n,\widehat{\,\,i_n})\otimes M_n$, we have
\begin{eqnarray*}
p_0\circ \Delta^u_\alpha =\delta^A_{\alpha}\circ p , \quad\quad  p_0: (a_0,a_1,f)\mapsto a_0\\
p_1\circ \Delta^u_\alpha =j^A_n\circ p, \quad\quad p_1: (a_0,a_1,f)\mapsto a_1
\end{eqnarray*}
 
\item Furthermore $[\Delta^u_\alpha\circ h]=\hat  h_*\Big( (  [\morph{A}{\alpha}] -[j^A_n])\otimes  [S\,i_{M,N}]\Big)\in KK(A\rtimes \Gamma\otimes SN,Z(i_n,\widehat{\,\,i_n}))$ where $S\,i_{M,N}$ is the inclusion $SN\hookrightarrow SM$ and where $h:A\rtimes \Gamma\otimes SN\to \mathcal{T}({\iota^A,\morph{A}{\lambda}})$ and $\hat h:A\rtimes \Gamma\otimes M_n(\C)\otimes SM\to A\rtimes \Gamma \otimes Z(i_n,\widehat{\,\,i_n})$ denote the inclusions $f\mapsto (0,f)$ and $f\mapsto (0,0,f)$.
\end{enumerate}
\end{prop}
\begin{proof}
$(1)$ The morphism is well defined. 
$(2)$ is immediate.
$(3)$ Obviously $\Delta^u_\alpha(A\rtimes \Gamma\otimes SN)\subset A\rtimes \Gamma\otimes M_n(\C)\otimes SM$. We just have to compute the $KK$-class of the morphism $\check \Delta^u_\alpha: A\rtimes \Gamma\otimes SN\to A\rtimes \Gamma\otimes M_n(\C)\otimes SM$ induced by $\Delta^u_\alpha$. 
Since the action of $\operatorname{Ad}_{u^*}$ is trivial in $K$-theory, we are computing the class of the sum $g_1+g_2$ where:
$$
g_1(t)=\begin{cases}
(\morph{A}{\alpha}\otimes i_{M,N})f(2t)& $for $t\le 1/2\\
0&$for $t\ge 1/2.
\end{cases}$$ and 
$$g_2(t)=\begin{cases}
0& $for $t\le 1/2\\
(j^A_n\otimes i_{M,N})f(2-2t)&$for $t\ge 1/2.
\end{cases}
$$
Both are morphisms from $A\rtimes \Gamma \otimes SN\longrightarrow A\rtimes \Gamma \otimes M_n(\C)\otimes SN$ composed with the inclusion $i_{M,N}$. Since the ranges of $g_1$ and $g_2$ are orthogonal, the sum as morphisms correspond to the sum of the $KK$-classes. Now the class of $g_2$ is the class of the morphism $S\delta^A_{\alpha}:A\rtimes \Gamma \otimes SN \longrightarrow A\otimes \Gamma \otimes SN$ induced on the suspension and  $[S\delta^A_{\alpha}]=1_{SN}\otimes [\delta^A_{\alpha}]$. The class of $g_1$ is easily seen to be the opposite of the class of the morphism $Sj_n^A$, for which again $[Sj_n^A]=1_{SN}\otimes [j_n^A].$
\end{proof}

\medskip In Remark \ref{commuT}, we constructed (using $u$) a bimodule yielding a Morita equivalence between the torus algebras $Z(i_n, {i_n})$ and $Z(i_n,\widehat{\,\,i_n})$. The $KK$-class  of this bimodule is an element in $KK(Z(i_n,\widehat{\,\,i_n}),Z(i_n, {i_n}))$. 
We denote by $[E_u]$ its image as a class in $KK(Z(i_n,\widehat{\,\,i_n}),Z(i, i))$, where we use the fact that $Z(i_n, {i_n})=M_n\otimes Z_{i,i}$, where $i$ is the inclusion $i: \C\to N$.

Define
\begin{equation}\label{psiM}
\psi_\alpha^M:= [\Delta_\alpha^u]\otimes [E_u]\in KK(\mathcal T, A\rtimes \Gamma\otimes Z_{i,i})\ .
\end{equation}

\begin{prop}
\label{prop2.4}
Denote by $ \tilde h:A\rtimes \Gamma\otimes M_n(\C)\otimes SM\to A\rtimes \Gamma \otimes M_n\otimes  Z(i,i)$  the inclusion $f\mapsto (0,0,f)$, and let $ p_0, p_1:A\rtimes \Gamma \otimes Z_{i,i}\to A\rtimes \Gamma$ be the natural evaluation maps.
\begin{enumerate}
\item  We have 
\begin{eqnarray}
(p_0)_*( \psi_\alpha^M)&=& [\morph{A}{\alpha}\circ  p], \\
 (p_1)_* (\psi_\alpha^M)&=&[j_n^A\circ  p]\nonumber\\
h^*( \psi_\alpha^M)&=& \tilde h_*\Big( ( [\morph{A}{\alpha}] - [j^A_n])\otimes  [S\,i_{M,N}]\Big).
\label{ev-htilde}
\end{eqnarray}
 \item $\psi_\alpha^M$ does not depend on $u$;
\item $\psi^M_\alpha$ behaves in a natural way with respect to inclusions $i_{\overline M, M}:M\hookrightarrow \overline M$ of ${\rm II}_1$-factors, namely
$$
j_{\overline M, M}(\psi_\alpha^M)=\psi_\alpha^{\overline M}
$$
where  $j_{\overline M, M}$ is the  map induced by $i_{\overline M, M}$ on the cylinders.
\end{enumerate}

\end{prop}

\begin{proof}
\begin{enumerate}
\item  The equality follows directly by Proposition \ref{prop2.3} and the first Remark \ref{commuT}.
\item This follows from the connectedness of the group of unitaries of a von Neumann algebra. In fact, another unitary $u'$ in $M_n(\C)\otimes M$  satisfying 
$(\alpha _g\otimes \lambda_g)u'=u'(1\otimes \lambda_g)$ is of the form $uv$, where $v$ is a unitary of the von Neumann algebra $(M_n(\C)\otimes M)\cap \{\lambda_g,\ g\in \Gamma\}' $. 

\item By part (2), we can construct $ \psi^{\overline M}_\alpha $ using the
  unitary $\bar u=(\id_{M_n(\C)} \otimes i_{\overline M, M})(u)$. 

  Let $\bar \iota: \C\to \overline M$ be the inclusion in $\overline M$, and $j_{\overline M, M}:Z(i,i)\to Z(\bar \iota, \bar \iota)$ the induced map on the cylinders. In this way, the Morita equivalence $[E_{\bar u}]$ used in the construction is compatible with the inclusion of cylinders. 
  \end{enumerate}
  \qedhere
\end{proof}

\begin{dfn}
We denote by $\sigma_\alpha^A$ the element 
$$\sigma_\alpha^A=\vartheta( \psi_\alpha^M)\in  KK^1_{\R/\Z}(\mathcal{T}({\iota^A,\morph{A}{\lambda}}),A\rtimes \Gamma)
$$ 
where $\vartheta:KK(\mathcal{T}({\iota^A,\morph{A}{\lambda}}),A\rtimes \Gamma\otimes Z_{i,i})\to KK^1_{\R/\Z}(\mathcal{T}({\iota^A,\morph{A}{\lambda}}),A\rtimes \Gamma)$ is the morphism described in Remark \ref{vartheta}.
\end{dfn}

\begin{prop}\label{prop2.5}
\begin{enumerate}
\item The class $\sigma_\alpha^A$ doesn't depend on the choices involved in this construction.
\item The image $\partial (\sigma_\alpha^A)\in KK(\mathcal{T}({\iota^A,\morph{A}{\lambda}}),A\rtimes \Gamma)$ through the connecting map of the  Bockstein exact sequence is  $\partial (\sigma_\alpha^A)=p^*([\morph{A}{\alpha}]-n1_{A\rtimes \Gamma})$.
\item The image $h^*(\sigma_\alpha^A)\in KK^1_{\R/\Z}(A\rtimes \Gamma\otimes SN,A\rtimes \Gamma)$ is the image of $([\morph{A}{\alpha}]-n1_{A\rtimes \Gamma})\otimes 1_{SN}\in KK(A\rtimes \Gamma\otimes SN,A\rtimes \Gamma\otimes SN)=KK^1(A\rtimes \Gamma\otimes SN,A\rtimes \Gamma \otimes N)$ through the composition 
$$KK^1(A\rtimes \Gamma\otimes SN,A\rtimes \Gamma \otimes N)\longrightarrow KK_{\R}^1(A\rtimes \Gamma\otimes SN,A\rtimes \Gamma)\longrightarrow KK^1_{\R/\Z}(A\rtimes \Gamma\otimes SN,A\rtimes \Gamma).
$$
\end{enumerate}
\begin{proof}
\begin{enumerate}
\item We know that $\psi^M_\alpha$ does not depend on the choice of $u$. Furthermore, it behaves well under trace preserving inclusions of ${\rm II}_1$-factors. Since $KK^1_{\R/\Z}(\cdot, \cdot)$ is defined by the limit \eqref{kcone} with respect to all these inclusions, every choice of $M$ defines the same class.

%
%
%
%
%

\item It follows from Remark \ref{vartheta} and Proposition \ref{prop2.3}, in fact: $\partial (\sigma^M_\alpha)=(\partial \circ \vartheta)(\psi^M_\alpha)=[\delta^A_\alpha\circ p]-[j_n^A\circ p]=p^*([\delta_\alpha^A]-1_{A\rtimes \Gamma})$, since $[j_n^A]=n1_{A\rtimes \Gamma}$.
\item It is a consequence of \eqref{ev-htilde}, because the inclusion $SN\hookrightarrow C_i$ induces the change of coefficients $\R\to \R/\Z$. \qedhere
\end{enumerate}
\end{proof}
\end{prop}

\begin{rem}
\label{remark2.6}
 Let $N_1\subset N$ be a sub-factor of $N$ containing $\lambda(\Gamma)$. The torus algebra $\mathcal{T}({\iota^A,\morph{A}{\lambda}})$ associated with $N_1$ is a subalgebra to the one associated with $N$. Obviously, the element $\sigma_\alpha^A$  associated with $N_1$ is just the restriction of the one of $N$.
%

\end{rem}

\subsection{Construction of $\hat\rho_{\alpha}^A$}

Let $A$ be a $\Gamma$-algebra satisfying \KFP, and $\alpha$ a finite dimensional unitary representation of $\Gamma$. 

By Remark \ref{prodtraces}, $[\morph{A}{\alpha}]_{\R}=n1_{A\rtimes \Gamma}^{\R}$ where $1_{A\rtimes \Gamma}^{\R}$ denotes the unit element of the ring $KK_{\R}(A\rtimes \Gamma,A\rtimes \Gamma)$. Using the Bockstein change of coefficients exact sequence, it follows that there exists $z\in KK^1_{{\R}/\Z}(A\rtimes \Gamma,A\rtimes \Gamma)$ whose image in $KK(A\rtimes \Gamma,A\rtimes \Gamma)$ is $[\morph{A}{\alpha}]-n1_{A\rtimes \Gamma}$.

We now construct a class $\hat\rho_{\alpha}^A\in KK^1_{{\R}/\Z}(A\rtimes \Gamma,A\rtimes \Gamma)$ whose image in $KK(A\rtimes \Gamma,A\rtimes \Gamma)$ is $[\morph{A}{\alpha}]-n1_{A\rtimes \Gamma}$.

The important fact in our construction, as for the one given in Definition \ref{defstrongrho} for the stronger context,  is that the element $\hat\rho_{\alpha}^A$ is \emph{independent of all choices.}

\bigskip 
Using the torus $\mathcal{T}({\iota^A, \morph{A}{\lambda}})$ and the corresponding exact sequence (see Section \ref{torus}), this means exactly that there exists a class
 $y\in KK(A\rtimes \Gamma, \mathcal{T}({\iota^A, \morph{A}{\lambda}}))$
 with $p_*(y)=1_{A\rtimes \Gamma}$ where $p$ is the projection $p:\mathcal{T}({\iota^A, \morph{A}{\lambda}}) \longrightarrow A\rtimes \Gamma$.

\begin{dfn}\label{defweakrho}
We denote by $\hat\rho_{\alpha}^A\in KK^1_{\R/\Z}(A\rtimes \Gamma,A\rtimes \Gamma)$ the product $y\otimes _{\mathcal{T}({\iota^A, \morph{A}{\lambda}})}\sigma_\alpha ^A$.
\end{dfn}

\begin{theorem}
\begin{enumerate}
\item The image of $\hat\rho_{\alpha}^A $ under the connecting map of the Bockstein change of coefficients exact sequence $KK^1_{\R/\Z}(A\rtimes \Gamma,A\rtimes \Gamma)\to KK(A\rtimes \Gamma,A\rtimes \Gamma)$ is $[\morph{A}{\alpha}]-n1_{A\rtimes \Gamma}$.
\item The class $\hat\rho_{\alpha}^A $ does not depend on the choices of $N$ and $y$ involved in its construction.
\end{enumerate}
\begin{proof}
 
\begin{enumerate}
\item We have $\partial \hat\rho_{\alpha}^A=y\otimes   _{\mathcal{T}({\iota^A, \morph{A}{\lambda}})}\partial(\sigma_\alpha ^A)=y\otimes   _{\mathcal{T}({\iota^A, \morph{A}{\lambda}})}p^*([\morph{A}{\alpha}]-n1_{A\rtimes \Gamma})$ by Proposition \ref{prop2.5}. Whence  $\partial \hat\rho_{\alpha}^A=p_*y\otimes _{A\rtimes \Gamma}([\morph{A}{\alpha}]-n1_{A\rtimes \Gamma})=[\morph{A}{\alpha}]-n1_{A\rtimes \Gamma}$ since $p^*y=1_{A\rtimes \Gamma}$.

\item Fix first $N$. We show that if $z\in KK(A\rtimes \Gamma, \mathcal{T}({\iota^A,\morph{A}{\lambda}}))$ satisfies $p_*(z)=0$ then $z\otimes _{\mathcal{T}({\morph{A}{\lambda},\iota^A})}\sigma_\alpha ^A=0$. Using the torus exact sequence, it is enough to show that for every $z\in KK(A\rtimes \Gamma, A\rtimes \Gamma\otimes SN)$ we have $h_*(z)\otimes _{\mathcal{T}({\iota^A, \morph{A}{\lambda}})}\sigma_\alpha ^A=z\otimes _{A\rtimes \Gamma\otimes SN}h^*(\sigma_\alpha ^A)=0$. But, as $A$ satisfies \KFP, $[\morph{A}{\alpha}]-n1_{A\rtimes \Gamma}$ is the $0$ element in $KK_\R(A\rtimes \Gamma,A\rtimes \Gamma)$ therefore, by Proposition \ref{prop2.5}, $h^*(\sigma_\alpha ^A)=0$. This shows that, given $N$, $\hat\rho_{\alpha}^A $ is independent on $y$.

Let $N_1$ be a ${\rm II}_1$-factor containing $N$.  Denote by $j:N\to N_1$ the inclusion. Put $\lambda_1=j\circ \lambda:C^*\Gamma \to N_1$, $\iota^1_A=(\id_{A\rtimes \Gamma}\otimes j):A\rtimes \Gamma \to A\rtimes \Gamma\otimes N_1$ and $p_1:\mathcal{T}({\iota^1_A,\morph{A}{\lambda_1}})\to A\rtimes \Gamma$ the morphism $(a,f)\mapsto a$. Put $y_1=j_*(y)\in \mathcal{T}({\iota^1_A,\morph{A}{\lambda_1}})$. It follows from remark \ref{remark2.6} that $(N_1,y_1)$ and $(N,y)$ define the same element $\hat\rho_{\alpha}^A$.
Using amalgamated free products, we then see that $\hat\rho_{\alpha}^A $ is also independent on $N$.
\qedhere\end{enumerate}
\end{proof}
\end{theorem}

\begin{rem}
If $A$ satifies \SKFP, it is in fact quite easy to compare the two constructions of $\rho_\alpha ^A$ of Definition \ref{defstrongrho} and the one we just constructed, by showing that $\hat\rho_{\alpha}^A =j_\Gamma(\rho_\alpha ^A)$, where $j_\Gamma:KK_{\R/\Z}^{\Gamma,1}(A,A)\to KK_{\R/\Z}^{1}(A\rtimes \Gamma,A\rtimes \Gamma)$ is Kasparov's descent morphism.

Indeed, the Kasparov descent of a Kasparov bimodule $(E,F)\in {\bf E}^\Gamma(A,A\otimes Z_i)$ joining $(A,0)$ and $(A\otimes P_\lambda)$ is an element $y\in KK(A\rtimes \Gamma, \mathcal{T}({\iota^A,\morph{A}{\lambda}}))$ as above; the gluing construction ``$\diamond$'' corresponds to the bimodule $\psi _\alpha^M$. It follows that the image by $j_\Gamma$ of  (the class of) the bimodule  $$(E\otimes V_{\alpha},F\otimes 1)\diamond _{\id_A\otimes u} (E\otimes V_n,F\otimes 1)$$ is $y\otimes _{\mathcal{T}({\iota^A,\morph{A}{\lambda}})}\psi _\alpha^M$. Taking the images in $KK^1_{\R/\Z}$, the result follows.
\end{rem}
\medskip

\providecommand{\bysame}{\leavevmode\hbox to3em{\hrulefill}\thinspace}
\providecommand{\MR}{\relax\ifhmode\unskip\space\fi MR }
\providecommand{\MRhref}[2]{%
  \url{http://www.ams.org/mathscinet-getitem?mr=#1}{#2}
}
\providecommand{\href}[2]{#2}


\providecommand{\bysame}{\leavevmode\hbox to3em{\hrulefill}\thinspace}
\providecommand{\MR}{\relax\ifhmode\unskip\space\fi MR }
\providecommand{\MRhref}[2]{%
  \url{http://www.ams.org/mathscinet-getitem?mr=#1}{#2}
}
\providecommand{\href}[2]{#2}

    \bigskip
    \bigskip

\noindent Paolo Antonini\\
D\'epartement de Math\'ematiques\\
B\^atiment 425 Facult\'e des Sciences d'Orsay,  Universit\'e Paris-Sud \\
F-91405 Orsay Cedex\\
e-mail: \texttt{paolo.antonini@math.u-psud.fr}

    \bigskip

\noindent Sara Azzali\\
Institut f\"ur Mathematik\\
Universit\"at Potsdam\\
Am Neuen Palais, 10\\
14469 Potsdam, Germany\\
e-mail: \texttt{azzali@uni-potsdam.de}

    \bigskip

\noindent Georges Skandalis\\
Universit\'e Paris Diderot, Sorbonne Paris Cit\'e\\
Sorbonne Universit\'es, UPMC Paris 06, CNRS, IMJ-PRG\\
UFR de Math\'ematiques, {\sc CP} {\bf 7012} - B\^atiment Sophie Germain \\
5 rue Thomas Mann, 75205 Paris CEDEX 13, France\\
e-mail: \texttt{skandalis@math.univ-paris-diderot.fr}

\end{document}